\documentclass[reqno,twoside,12pt]{article}
 \usepackage[english]{babel}
 \usepackage{times}
 \usepackage[
    paper=a4paper,
    portrait=true,
    textwidth=425pt,
    textheight=650pt,
    tmargin=3cm,
    marginratio=1:1
            ]{geometry}
 \usepackage[centertags]{amsmath}
 \usepackage{amsfonts}
 \usepackage{amsthm}
 \usepackage{amssymb}
 \usepackage{enumerate}
 \usepackage{mathrsfs}      %(\mathcal: normal script; \mathscr: fancy script)%

 \theoremstyle{plain}
 \newtheorem{Thm}{Theorem}[section]
 \newtheorem{Cor}[Thm]{Corollary}
 \newtheorem{Lemma}[Thm]{Lemma}
 \newtheorem{Prop}[Thm]{Proposition}

 \theoremstyle{definition}
 \newtheorem{Rem}[Thm]{Remark}
 
 \newtheorem{Defi}[Thm]{Definition}
 \numberwithin{Thm}{section}
 \numberwithin{equation}{section}

\def\qquad{\quad\quad}

\def\msy#1{{\mathbb #1}}
\def\C{{\msy C}}

\def\H{{\msy H}}
\def\F{{\msy F}}

\def\R{{\msy R}}

\def\fa{{\mathfrak a}}
\def\fb{{\mathfrak b}}
\def\fg{{\mathfrak g}}
\def\fh{{\mathfrak h}}

\def\fk{{\mathfrak k}}
\def\fl{{\mathfrak l}}

\def\fn{{\mathfrak n}}
\def\fp{{\mathfrak p}}
\def\fq{{\mathfrak q}}

\def\cA{{\mathcal A}}

\def\cC{{\mathcal C}}
\def\cD{{\mathcal D}}

\def\cF{{\mathcal F}}
\def\cH{{\mathcal H}}

\def\cP{{\mathcal P}}
\def\cQ{{\mathcal Q}}
\def\cR{{\mathcal R}}
\def\cS{{\mathcal S}}

\def\cU{{\mathcal U}}

\def\cW{{\mathcal W}}

\newcommand{\sS}{\mathscr{S}}

\def\Ht{\cH}
\def\Ft{\cF}

\def\Rt{\cR}

\def\to{\rightarrow}
\def\Re{\mathrm{Re}\,}

\def\Ad{\mathrm{Ad}}
\def\End{\mathrm{End}}

\def\Hom{\mathrm{Hom}}

\def\ad{\mathrm{ad}}

\def\tr{\mathrm{tr}\,}

\def\supp{\mathop{\rm supp}}
\def\Lie{\mathop{\rm Lie}}
\def\Ind{\mathrm{Ind}}
\def\dotvar{\, \cdot\,}

\def\diag{\mathrm{diag}}

\def\1{\mathbf{1}}

\def\temp{\mathrm{temp}}

\def\bs{\backslash}

\def\GL{\mathrm{GL}}
\def\SL{\mathrm{SL}}
\def\PSL{\mathrm{PSL}}
\def\PSU{\mathrm{PSU}}
\def\SO{\mathrm{SO}}
\def\SU{\mathrm{SU}}
\def\U{\mathrm{U}}
\def\Sp{\mathrm{Sp}}
\def\Orth{\mathrm{O}}

 \title{Cuspidal integrals and subseries for $\SL(3)/K_{\epsilon}$}
 \author{
    Mogens~Flensted--Jensen,
    Job~J.~Kuit
    }
 \date{\today}

\begin{document}
 \maketitle
 \tableofcontents
%%% ----------------------------------------------------------------------
%%% ----------------------------------------------------------------------
\begin{abstract}
We show that for the symmetric spaces $\SL(3,\R)/\SO(1,2)_{e}$ and $\SL(3,\C)/\SU(1,2)$ the cuspidal integrals are absolutely convergent. We further determine the behavior of the corresponding Radon transforms and relate the kernels of the Radon transforms to the different series of representations occurring in the Plancherel decomposition of these spaces. Finally we show that for the symmetric space $\SL(3,\H)/\Sp(1,2)$ the cuspidal integrals are not convergent for all Schwartz functions.
\end{abstract}
%%% ----------------------------------------------------------------------
%%% ----------------------------------------------------------------------
\section*{Introduction}\addcontentsline{toc}{section}{Introduction}
%%% ----------------------------------------------------------------------
In this article we investigate the notion of cusp forms for some symmetric spaces of split rank $2$.

Harish-Chandra defined a notion of cusp forms for reductive Lie groups and he showed that the space of cusp forms coincides with the closure in the Schwartz space of the span of the discrete series of representations. This fact plays an important role in his work on the
Plancherel decomposition for reductive groups.
In \cite{AndersenFlenstedJensenSchlichtkrull_CuspidalDiscreteSerieseForSemisimpleSymmetricSpaces} and \cite{AndersenFlensted-Jensen_CuspidalDiscreteSeriesForProjectiveHyperbolicSpaces} a notion of cusp forms was suggested for reductive symmetric spaces, specifically for hyperbolic spaces. This notion was adjusted in \cite{vdBanKuit_HC-TransformAndCuspForms} to a notion for reductive symmetric spaces of split rank $1$.

The main problem one encounters when trying to define cusp forms, is convergence of the integrals involved. For a reductive symmetric space $G/H$ of split rank $1$, this problem was solved in \cite{vdBanKuit_HC-TransformAndCuspForms} by identifying a class of parabolics subgroups $P$ of $G$, the so-called $\fh$-compatible parabolic subgroups, for which the integrals
\begin{equation}\label{eq N-integral}
\int_{N_{P}/N_{P}\cap H}\phi(n)\,dn
\end{equation}
are absolutely convergent for all Schwartz functions $\phi$ on $G/H$. In \cite{vdBanKuitSchlichtkrull_SLn} it was shown that for the spaces $\SL(n,\R)/\GL(n-1,\R)$ the condition of $\fh$-compatibility is necessary for the integrals (\ref{eq N-integral}) to converge for all Schwartz functions $\phi$. Cusp forms are then defined to be those Schwartz functions $\phi$ for which
$$
\int_{N_{P}/N_{P}\cap H}\phi(gn)\,dn=0
$$
for every $\fh$-compatible parabolic subgroup $P$ and every $g\in G$.
For reductive symmetric spaces of split rank larger than $1$, the methods in \cite{vdBanKuit_EistensteinIntegrals} cannot be used to show convergence of the integrals.

For the spaces of type $G/K_{\epsilon}$, which are described in \cite{OshimaSekiguchi_EigenspacesOfInvariantDifferentialOperators}, the condition of $\fh$-compatibility is void.
If $\sigma$ denotes the involution determining the symmetric subgroup $K_{\epsilon}$, then the naive definition of cusp form involves the integrals (\ref{eq N-integral}) for all $\sigma$-parabolic subgroups $P\neq G$, i.e., all parabolic subgroups $P\neq G$ such that $\sigma(P)$ is opposite to $P$. This would require the integrals (\ref{eq N-integral}) to converge for all $\sigma$-parabolic subgroups $P$ and all Schwartz functions $\phi$.

In this article we investigate the convergence of such integrals for three reductive symmetric spaces of split rank $2$ of the type described in \cite{OshimaSekiguchi_EigenspacesOfInvariantDifferentialOperators}, namely $\SL(3,\R)/\SO(1,2)_{e}$, $\SL(3,\C)/\SU(1,2)$ and $\SL(3,\H)/\Sp(1,2)$. For the first two of these spaces we show that for all $\sigma$-parabolic subgroups the integrals are absolutely convergent. We further show how one can characterize the different series of representation occurring in the Plancherel decomposition of these spaces with the use of these integrals.

Using the higher-rank analogue of \cite[Section 7.2]{vdBanKuit_HC-TransformAndCuspForms} and a careful analysis of the residues occurring in the analogue of the formula in \cite[Lemma 7.8]{vdBanKuit_HC-TransformAndCuspForms} for the space $\SL(3,\H)/\Sp(1,2)$, Erik van den Ban showed in unpublished notes that the integrals (\ref{eq N-integral}) are not converging for all minimal $\sigma$-parabolic subgroups of $\SL(3,\H)$ and all Schwartz functions. We give a short and easy argument showing that not even for the maximal $\sigma$-parabolic subgroups all of the integrals are converging.
The non-convergence of the integrals for this space raises the question whether it is possible to give a useful definition for cusp forms for reductive symmetric spaces of split rank larger than $1$.

The article is organized as follows.
In Section \ref{Section Setup} we describe the structure of the $3$ above mentioned symmetric spaces, their Schwartz spaces and the parabolic subgroups. For the spaces $\SL(3,\R)/\SO(1,2)_{e}$ and $\SL(3,\C)/\SU(1,2)$ we then show in Section \ref{Section Convergence} that all integrals are convergent and in Section \ref{Section rep theory} we make the connection with the Plancherel decompositions of these spaces. Finally, in Section \ref{Section Divergence} we prove that there exists a Schwartz-function on $\SL(3,\H)/\Sp(1,2)$ such that the integrals (\ref{eq N-integral}) are divergent for some of the maximal and all of the minimal parabolic subgroups $P$.

This paper grew out of discussions with Erik van den Ban and Henrik Schlichtkrull about explicit computations for a simple split rank $2$ symmetric space.
We want to thank both of them for their contribution through these discussions. In particular we want to thank Erik van den Ban for allowing us to publish our simple proof of his result on the non-convergence.

%%% ----------------------------------------------------------------------
\section{Structure, parabolic subgroups and Schwartz spaces}
\label{Section Setup}
%%% ----------------------------------------------------------------------
%%% ----------------------------------------------------------------------
\subsection{Involutions}
%%% ----------------------------------------------------------------------
Let $\F\in\{\R,\C,\H\}$ and let $G=\SL(3,\F)$.
Let $\theta$ be the usual Cartan involution
$$
\theta:g\mapsto (g^{-1})^{\dagger}
$$
and let $\sigma$ be the involution
$$
\sigma:g\mapsto \left(
                  \begin{array}{ccc}
                    1 & 0 & 0 \\
                    0 & -1 & 0 \\
                    0 & 0 & -1 \\
                  \end{array}
                \right)
                \theta(g)
                \left(
                  \begin{array}{ccc}
                    1 & 0 & 0 \\
                    0 & -1 & 0 \\
                    0 & 0 & -1 \\
                  \end{array}
                \right).
$$
We define $H$ to be the connected open subgroup of $G^{\sigma}$ and $K$ to be the maximal compact subgroup $G^{\theta}$.
Note that
\begin{itemize}
\item $H=\SO(1,2)_{e}$ and $K=\SO(3)$ if $\F=\R$;
\item $H=\SU(1,2)$ and $K=\SU(3)$ if $\F=\C$;
\item $H=\Sp(1,2)$ and $K=\Sp(3)$ if $\F=\H$.
\end{itemize}
The involutions $\theta$ and $\sigma$ commute. We use the same symbols for the involutions of $\fg$ obtained by deriving $\theta$ and $\sigma$.
Let $\fp$ and $\fq$ be the $-1$ eigenspaces of $\theta$ and $\sigma$ respectively. Then
$$
\fg
=\fk\oplus\fp
=\fh\oplus\fq,
$$
where $\fk$ and $\fh$ are the Lie algebras of $K$ and $H$ respectively.

%%% ----------------------------------------------------------------------
\subsection{$\sigma$-Stable maximal split abelian subalgebras}
%%% ----------------------------------------------------------------------
Let $\fb$ be a $\sigma$-stable maximal split abelian subalgebra of $\fg$.
The split rank of $G$ is equal to $2$, while the split rank of $H$ is equal to $1$. Therefore $\dim(\fb)=2$ and $\dim(\fb\cap\fh)\leq 1$. We define $\cA$ to be the set of all $\sigma$-stable maximal split abelian subalgebra $\fb$ such that $\dim(\fb\cap\fh)=0$, i.e., $\fb\subset\fq$. Note that $H$ acts on $\cA$.

\begin{Prop}\label{Prop H-conj classes of sigma-stable max split subalgebras}
The action of $H$ on $\cA$ is transitive.
\end{Prop}

\begin{proof}
The proposition follows directly from \cite[Lemma 1.1]{vdBanKuitSchlichtkrull_SLn} and \cite[Lemmas 4\&7]{Matsuki_TheOrbitsOfAffineSymmetricSpacesUnderTheActionOfMinimalParabolicSubgroups}.
\end{proof}

Let $\fa$ be the Lie subalgebra of $\fg$ consisting of all diagonal matrices. Then $\fa\subset \fp\cap \fq$ and thus $\fa\in\cA$. It follows from Proposition \ref{Prop H-conj classes of sigma-stable max split subalgebras} that every maximal split abelian subalgebra in $\cA$ is conjugate to $\fa$ via an element in $H$.

For later purposes we note here that the group $N_{K\cap H}(\fa)/Z_{K\cap H}(\fa)$ consists of two elements: the equivalence class of the unit element and the equivalence class of
\begin{equation}\label{eq matrix representing non-trivial element in W_(K cap H)}
w_{0}=\left(
    \begin{array}{ccc}
      1 & 0 & 0 \\
      0 & 0 & 1 \\
      0 & -1 & 0 \\
    \end{array}
  \right).
\end{equation}

%%% ----------------------------------------------------------------------
\subsection{$H$-conjugacy classes of minimal $\sigma$-parabolic subgroups}
%%% ----------------------------------------------------------------------
Let $\Sigma$ be the root system of $\fa$ in $\fg$. For $i\in\{1,2,3\}$ let $e_{i}\in\fa^{*}$ be given by
$$
e_{i}\big(\diag(x_{1},x_{2},x_{3})\big)
=x_{i}.
$$
Then
$$
\Sigma=\{e_{i}-e_{j}:1\leq i,j\leq 3, i\neq j\}.
$$
We write $\fg_{\alpha}$ for the root space of a root $\alpha\in\Sigma$.

We define
\begin{align}\label{eq def Sigma_i}
\nonumber\Sigma_{1}&:=\{e_{1}-e_{2},e_{1}-e_{3},e_{2}-e_{3}\},\\
\Sigma_{2}&:=\{e_{2}-e_{1},e_{1}-e_{3},e_{2}-e_{3}\},\\
\nonumber\Sigma_{3}&:=\{e_{2}-e_{1},e_{3}-e_{1},e_{2}-e_{3}\}.
\end{align}
Note that $\Sigma_{i}$, with $i\in\{1,2,3\}$, is a positive system of $\Sigma$.
We define $P_{i}$ to be the minimal parabolic subgroup such that $\fa\subset\Lie(P_{i})$ and the set of roots of $\Lie(P_{i})$ in $\fa$ is equal to $\Sigma_{i}$.

\begin{Prop}\label{Prop characterization min sigma-psgs}
Let $P$ be a minimal parabolic subgroup of $G$. The following are equivalent.
\begin{enumerate}[(i)]
\item $P$ is a $\sigma$-parabolic subgroup, i.e., $\sigma P$ is opposite to $P$;
\item $PH$ is open in $G$;
\item There exists a $\fb\in\cA$ such that $\fb\subset\Lie(P)$;
\item $P$ is $H$-conjugate to one of the parabolic subgroups $P_{1}$, $P_{2}$ or $P_{3}$.
\end{enumerate}
\end{Prop}

\begin{proof}
$(i)\Rightarrow(ii)$:
Assume that $\sigma P$ is opposite to $P$. Then
\begin{equation}\label{eq Lie P + fh=fg}
\Lie(P)+\fh
\supseteq\Lie(P)+\frac{1+\sigma}{2}(\Lie(P))
=\Lie(P)+\sigma \Lie(P)
=\fg.
\end{equation}
This proves that $PH$ is open in $G$.\\
$(ii)\Rightarrow(iii)$:
Every element in $\cA$ is a subspace that is contained in $\fq$, and vice versa, every maximal split abelian subspace of $\fq$ is an element of $\cA$. The implication now follows from \cite[Lemma 14]{Rossmann_TheStructureOfSemisimpleSymmetricSpaces}.\\
$(iii)\Rightarrow(iv)$:
By Proposition \ref{Prop H-conj classes of sigma-stable max split subalgebras} we may without loss of generality assume that $\fa\subset\Lie(P)$. Let $k=w_{0}$ if $\fg_{3,2}\subset\Lie(P)$ (see (\ref{eq matrix representing non-trivial element in W_(K cap H)})); otherwise, let $k=e$. Note that in both cases $k\in N_{K\cap H}(\fa)$. Therefore $P':=kPk^{-1}$ is a parabolic subgroup such that $\fa\subset\Lie(P')$. Moreover, $\fg_{2,3}\subset \Lie(P')$, hence $P'$ is equal to one of the parabolic subgroups $P_{1}$, $P_{2}$ or $P_{3}$.\\
$(iv)\Rightarrow(i)$:
The parabolic subgroups $P_{i}$ for $i\in\{1,2,3\}$ are stable under the involution $\sigma\theta$. Therefore they are all $\sigma$-parabolic subgroups. Any $H$-conjugate of a $\sigma$-parabolic subgroup is again a $\sigma$-parabolic subgroup, hence $P$ is a $\sigma$-parabolic subgroup.
\end{proof}

Let $\cP_{\sigma}$ be the set of all minimal $\sigma$-parabolic subgroups of $G$. Note that $H$ acts on $\cP_{\sigma}$.

\begin{Prop}\label{Prop H-conj classes of min sigma-psgs}
The action of $H$ on $\cP_{\sigma}$ admits three orbits. Moreover,
$$
H\bs\cP_{\sigma}
=\{[P_{i}]:i=1,2,3\},
$$
where $[P]$ denotes the $H$-conjugacy class of $P$.
\end{Prop}

\begin{proof}
Let $P\in\cP_{\sigma}$.
Since $N_{K\cap H}(\fa)/Z_{K\cap H}(\fa)$ consists of two elements, it follows from \cite[Corollary 17]{Rossmann_TheStructureOfSemisimpleSymmetricSpaces} (see also \cite[Proposition 1]{Matsuki_TheOrbitsOfAffineSymmetricSpacesUnderTheActionOfMinimalParabolicSubgroups}) that there are three open $H$-orbits in $G/P$. In other words, there are three $H$-conjugacy classes of parabolic subgroups $P'$ such that $P'H$ is open in $G$. The proposition now follows from Proposition \ref{Prop characterization min sigma-psgs}.
\end{proof}

We conclude this section with a relation between $P_{1}$ and $P_{3}$. Recall $w_{0}\in N_{K\cap H}(\fa)$ from (\ref{eq matrix representing non-trivial element in W_(K cap H)}).

\begin{Prop}\label{Prop relations between P_i's}
$\sigma(w_{0}P_{1}w_{0}^{-1})=P_{3}.$
\end{Prop}

\begin{proof}
Note that $\fa\subset\Lie (\sigma(w_{0}P_{1}w_{0}^{-1}))=\Ad(w_{0})\big(\sigma\Lie(P_{1})\big)$. Moreover, the set of roots of $\fa$ in $Ad(w_{0})\big(\sigma\Lie(P_{1})\big)$ is equal to $\Sigma_{3}$. This proves the proposition.
\end{proof}

%%% ----------------------------------------------------------------------
\subsection{$H$-conjugacy classes of maximal $\sigma$-parabolic subgroups}
%%% ----------------------------------------------------------------------
For $i\in\{1,2,3,4\}$ let $Q_{i}$ be the maximal parabolic subgroup such that $\fa\subset\Lie Q$ and the nilradical $\fn_{i}$ of $\Lie Q_{i}$ is given by
$$
\fn_{1}=\fg_{1,3}\oplus\fg_{2,3},\quad
\fn_{2}=\fg_{2,1}\oplus\fg_{3,1},\quad
\fn_{3}=\fg_{1,2}\oplus\fg_{1,3},\quad
\fn_{4}=\fg_{2,1}\oplus\fg_{2,3}.
$$

\begin{Prop}\label{Prop characterization max sigma-psgs}
Let $Q$ be a maximal parabolic subgroup. The following are equivalent.
\begin{enumerate}[(i)]
\item $Q$ is a $\sigma$-parabolic subgroup, i.e., $\sigma Q$ is opposite to $Q$;
\item $QH$ is open in $G$;
\item $Q$ contains a minimal parabolic subgroup $P$ such that $PH$ is open;
\item $Q$ is $H$-conjugate to one of the parabolic subgroups $Q_{1}$, $Q_{2}$, $Q_{3}$ or $Q_{4}$.
\end{enumerate}
\end{Prop}

\begin{proof}
$(i)\Rightarrow(ii)$:
Assume that $\sigma Q$ is opposite to $Q$. Then (\ref{eq Lie P + fh=fg}) holds with $P$ replaced by $Q$ and thus it follows that $QH$ is open in $G$.\\
$(ii)\Rightarrow(iii)$:
Assume that $QH$ is open in $G$. Every minimal parabolic subgroup has only finitely many orbits in $G/H$. Therefore there exists a minimal parabolic subgroup $P\subset Q$ such that $PH$ is open.\\
$(iii)\Rightarrow(iv)$:
Assume that $Q$ contains a minimal parabolic subgroup $P$ such that $PH$ is open. From Proposition \ref{Prop characterization min sigma-psgs} it follows that $Q$ is $H$-conjugate to a maximal parabolic subgroup containing one of the minimal parabolic subgroups $P_{1}$, $P_{2}$ or $P_{3}$. Any maximal parabolic subgroup containing one of these three minimal parabolic subgroups is equal to $Q_{1}$, $Q_{2}$, $Q_{3}$ or $Q_{4}$.
\\
$(iv)\Rightarrow(i)$:
The parabolic subgroups $Q_{i}$ for $i\in\{1,2,3,4\}$ are stable under the involution $\sigma\theta$. Therefore they are all $\sigma$-parabolic subgroups. Any $H$-conjugate of a $\sigma$-parabolic subgroup is again a $\sigma$-parabolic subgroup, hence $Q$ is a $\sigma$-parabolic subgroup.
\end{proof}

Let $\cQ_{\sigma}$ be the set of all maximal $\sigma$-parabolic subgroups of $G$. Note that $H$ acts on $\cQ_{\sigma}$.

\begin{Prop}\label{Prop H-conj classes of max sigma-psgs}
The action of $H$ on $\cQ_{\sigma}$ admits four orbits. Moreover,
$$
H\bs\cQ_{\sigma}
=\{[Q_{i}]:i=1,2,3,4\},
$$
where $[Q]$ denotes the $H$-conjugacy class of $Q$.
\end{Prop}

\begin{proof}
There are two $G$-conjugacy classes of maximal parabolic subgroups. Let $[Q]_{G}$ denote the $G$-conjugacy class of $Q$. Then
$$
[Q_{1}]_{G}=[Q_{2}]_{G}\neq[Q_{3}]_{G}=[Q_{4}]_{G}.
$$
From \cite[Corollary 16 (2)]{Rossmann_TheStructureOfSemisimpleSymmetricSpaces} it easily seen that $Q_{i}$ admits two open orbits in $G/H$ for $i\in\{1,2,3,4\}$. (From the proof of the corollary it follows that one may take the group $A$ in the theorem to be equal to $\exp(\fa)$.) In view of Proposition \ref{Prop characterization max sigma-psgs} it follows that there are two $H$-conjugacy classes of $\sigma$-parabolic subgroups in each $G$-conjugacy class $[Q_{i}]_{G}$, hence in total there are four $H$-conjugacy classes in $\cQ_{\sigma}$. The remaining claim follows from Proposition \ref{Prop characterization max sigma-psgs}.
\end{proof}

Recall $w_{0}\in N_{K\cap H}(\fa)$ from (\ref{eq matrix representing non-trivial element in W_(K cap H)}).

\begin{Prop}\label{Prop relations between Q_i's}
$\sigma Q_{2}=Q_{3}$
and $\sigma(w_{0}Q_{1}w_{0}^{-1})=Q_{4}$.
\end{Prop}

\begin{proof}
Since $\fa$ is contained in $\Lie(Q_{i})$ for all $i\in\{1,2,3,4\}$ and both $\sigma$ and $\Ad(w_{0})$ stabilize $\fa$, it suffices to show that
$$
\sigma \fn_{2}=\fn_{3},\quad
\sigma\Ad(w_{0})\fn_{1}=\fn_{4}.
$$
The latter follows from a simple computation.
\end{proof}

%%% ----------------------------------------------------------------------
\subsection{Polar decomposition, the Schwartz space and tempered functions}\label{subsection Polar decomposition and Schwartz spaces}
%%% ----------------------------------------------------------------------
In this section we discuss the polar decomposition for $G/H$ and sub-symmetric spaces of $G/H$. We further give a definition of Harish-Chandra Schwartz functions and tempered functions.

Let $L$ be a $\sigma$-stable closed subgroup of $G$. Assume that $L$ is a reductive Lie group of the Harish-Chandra class.
Let $\theta_{L}$ be a Cartan involution of $L$ that commutes with $\sigma$ and let $K_{L}=L^{\theta_{L}}$ be the corresponding $\sigma$-stable maximal compact subgroup of $L$. Let $\fa_{L}$ be a maximal split abelian subalgebra of $\fl$ contained in $\fl\cap\fq$ and let $A_{L}=\exp(\fa_{L})$. Finally let $H_{L}$ be the symmetric subgroup $L^{\sigma}=H\cap L$ of $L$.
(Note that $L=G$, $K_{L}=K$, $A_{L}=A$ and $H_{L}=H$ are valid choices for the above defined subgroups.)

The space $L/H_{L}$  admits a polar decomposition: the map
$$
K_{L}\times A_{L}\to L/H_{L};
\qquad
(k,a)\mapsto ka\cdot H_{L}
$$
is surjective. Moreover, if $a\cdot H_{L}\in K_{L} a'\cdot H_{L}$ with $a_,a'\in A_{L}$, then there exists an element $k\in N_{K_{L}\cap H_{L}}(\fa_{L})$ such that $a=ka'k^{-1}$.

Let $P_{L}$ be a minimal $\sigma$-parabolic subgroup of $L$. We define $\rho_{P_{L}}\in\fa_{L}^{*}$ by
$$
\rho_{P_{L}}(Y)
:=\frac{1}{2}\tr\big(\ad(Y)\big|_{\fn_{P_{L}}}\big)
\qquad(Y\in\fa_{L}).
$$
Let $W_{L}$ be the Weyl group of the root system in $\fa_{L}$.

\begin{Defi}\label{defi cC}
A Schwartz function on $L/H_{L}$ is a smooth function $\phi:L/H_{L}\to\C$, such that for every $u\in\cU(\fl)$ and $r\geq0$ the seminorm
$$
\mu^{L}_{u,r}(\phi)
:=\sup_{k\in K_{L}}\sup_{a\in A_{L}}\Big(\sum_{w\in W_{L}}a^{w\cdot\rho_{P_{L}}}\Big)\big(1+\|\log(a)\|\big)^{r}
    \big|(u\phi)(ka\cdot H_{L})\big|
$$
is finite. Here the action of $\cU(\fl)$ on $C^{\infty}(L/H_{L})$ is obtained from the left-regular representation of $L$ on $C^{\infty}(L/H_{L})$.
We denote the vector space of Schwartz functions on $L/H_{L}$ by $\cC(L/H_{L})$ and equip $\cC(L/H_{L})$ with the topology induced by the mentioned seminorms.
Our definition of Schwartz functions is equivalent to the one in  \cite[section 17]{vdBan_PrincipalSeriesII}.

We further define $C^{\infty}_{\temp}(L/H_{L})$ to be the space of smooth functions on $L/H_{L}$ which are tempered as distributions on $L/H_{L}$, i.e., belong to the dual $\cC'(L/H_{L})$ of the Schwartz space $\cC(L/H_{L})$.
We equip the space $C^{\infty}_{\temp}(L/H_{L})$ with the coarsest locally convex topology such that the inclusion maps into $C^{\infty}(L/H_{L})$ and $\cC'(L/H_{L})$ are both continuous. Here $C^\infty(L/H_{L})$ is equipped with the usual Fr\'echet topology and $\cC'(L/H_{L})$ is equipped with the strong dual topology.
\end{Defi}

We finish this section with a more precise description of $\cC(G/H)$.
We define $\Phi: G\to\R_{>0}$ by
\begin{equation}\label{eq def Phi}
\Phi(g)
=\|g\sigma(g)^{-1}\|_{HS}^{2}\, \|\sigma(g)g^{-1}\|_{HS}^{2}
\qquad(g\in G),
\end{equation}
where $\|\cdot\|_{HS}$ denotes the Hilbert-Schmidt norm on $\textnormal{Mat}(n,\F)$.
We define
$$
V
:=\{t\in\R^{3}:\sum_{i=1}^{3}t_{i}=0\}.
$$
For $t\in V$ we further define
\begin{equation}\label{eq def a_t}
a_{t}
:=\left(
    \begin{array}{ccc}
      e^{t_{1}}& 0 & 0 \\
      0 & e^{t_{2}} & 0 \\
      0 & 0 & e^{t_{3}} \\
    \end{array}
  \right).
\end{equation}

\begin{Lemma}\label{Lemma KAH decomposition}
Let $g\in G$ and $t\in V$. If $g\in Ka_{t}\cdot H$, then
\begin{align}\label{eq expression Phi(g)}
\Phi(g)
&=\Big(\sum_{i=1}^{3}e^{4t_{i}}\Big)\Big(\sum_{i=1}^{3}e^{-4t_{i}}\Big)\\
\nonumber&=3+2\cosh\big(4(t_{1}-t_{2})\big)+2\cosh\big(4(t_{1}-t_{3})\big)+2\cosh\big(4(t_{2}-t_{3})\big).
\end{align}
In particular, $\Phi$ is left $K$-invariant, right $H$-invariant and $\Phi\circ\sigma=\Phi$. Moreover, $\Phi|_{A}$ is invariant under the action of the Weyl group.
\end{Lemma}

\begin{proof}
A straightforward computation shows that
$$
\|g\sigma(g)^{-1}\|_{HS}^{2}
=\tr\Big(g\sigma(g)^{-1}\big(g\sigma(g)^{-1}\big)^{t}\Big)
=\tr(a_{t}^{4})
$$
and
$$
\|\sigma(g)g^{-1}\|_{HS}^{2}
=\tr\Big(\sigma(g)g^{-1}\big(\sigma(g)g^{-1}\big)^{t}\Big)
=\tr(a_{t}^{-4}).
$$
The equalities in (\ref{eq expression Phi(g)}) follow from (\ref{eq def a_t}). Equation (\ref{eq expression Phi(g)}) may be rewritten as
$$
\Phi(ka\cdot H)
=3+\sum_{\alpha\in\Sigma}\cosh\big(4\alpha(\log a)\big)\qquad(k\in K, a\in A).
$$
From this identity the claimed invariances are clear.
\end{proof}

\begin{Rem}\label{Rem Phi geq 9}
From (\ref{eq expression Phi(g)}) it follows that $\Phi(g)\geq 9$ for all $g\in G$.
\end{Rem}

Let $k=\dim_{\R}\F$. Let $\rho_{1}$ be the half the sum of the roots in $\Sigma_{1}$, see (\ref{eq def Sigma_i}), and let $x\in \fa$ be in the corresponding positive Weyl chamber. In view of Lemma \ref{Lemma KAH decomposition}
$$
3+e^{\frac{4}{k}\rho_{1} x}
\leq\Phi(\exp x)
\leq3+3e^{\frac{4}{k}\rho_{1} x}.
$$
The Weyl group invariance of $\Phi\big|_{A}$ now implies that a smooth function $\phi:G/H\to\C$ belongs to the Schwartz space $\cC(G/H)$ if and only if for every $u\in\cU(\fg)$ and $r\geq0$ the seminorm
$$
\mu_{u,r}(\phi)
:=\sup_{x\in G/H}\Phi(x)^{\frac{k}{4}}\big(\log\circ\Phi(x)\big)^{r}\big|(u\phi)(x)\big|
$$
is finite.

%%% ----------------------------------------------------------------------
\section{Convergence of cuspidal integrals}
\label{Section Convergence}
%%% ----------------------------------------------------------------------
Throughout this section, let $\F=\R$ or $\F=\C$.

%%% ----------------------------------------------------------------------
\subsection{Main theorem}
%%% ----------------------------------------------------------------------
\begin{Thm}\label{Thm convergence and behavior for sigma parabolics}
Let $P$ be a $\sigma$-parabolic subgroup and let $P=M_{P}A_{P}N_{P}$ be a Langlands decomposition of $P$ so that $M_{P}$ and $A_{P}$ are $\sigma$-stable. We set $L_{P}:=M_{P}A_{P}=P\cap\sigma(P)$.
\begin{enumerate}[(i)]
\item\label{Thm convergence and behavior for sigma parabolics item 1}
For every $\phi\in\cC(G/H)$ and $g\in G$ the integral
\begin{equation}\label{eq def Rt_P}
\Rt_{P}\phi(g)
:=\int_{N_{P}}\phi(gn)\,dn
\end{equation}
is absolutely convergent and the function $\Rt_{P}\phi$ thus obtained is a smooth function on $G/(L_{P}\cap H)N_{P}$.
\item\label{Thm convergence and behavior for sigma parabolics item 2}
Let $\delta_{P}$ be the character on $P$ given by
$$
\delta_{P}(l)
:=|Ad(l)\big|_{\Lie(P)}|^{\frac{1}{2}}
=|Ad(l)\big|_{\fn_{P}}|^{\frac{1}{2}}
\qquad(l\in L_{P}).
$$
Define for $\phi\in\cC(G/H)$ the function $\Ht_{P}\phi\in C^{\infty}(L_{P})$ by
$$
\Ht_{P}\phi(l)
:=\delta_{P}(l)\Rt_{P}\phi(l)\qquad(l\in L_{P}).
$$
Then $\Ht_{P}\phi$ is right $L_{P}\cap H$-invariant and $\Ht_{P}$ defines a continuous linear map $\cC(G/H)\to C_{\temp}^{\infty}\big(L_{P}/L_{P}\cap H\big)$. Moreover,
\begin{enumerate}[a.]
\item if $\F=\R$, then $\Ht_{P}$ defines a continuous linear map $\cC(G/H)\to \cC\big(L_{P}/L_{P}\cap H\big)$;
\item if $\F=\C$, then $\phi\mapsto \Ht_{P}\phi\big|_{M_{P}}$ defines a continuous linear map $\cC(G/H)\to \cC\big(M_{P}/M_{P}\cap H\big)$.
\end{enumerate}
\end{enumerate}
\end{Thm}

In the remainder of section \ref{Section Convergence} we give the proof for Theorem \ref{Thm convergence and behavior for sigma parabolics}.

%%% ----------------------------------------------------------------------
\subsection{Some estimates}\label{Subsection estimates}
%%% ----------------------------------------------------------------------
We define the functions
$$
M:\R\to[9,\infty);\qquad x\mapsto\max(9,x)
$$
and
$$
L:=\log\circ M:\R\to[\log 9,\infty).
$$
Note that $M$ and $L$ are monotonically increasing.

\begin{Lemma}\label{Lemma integral inequalities}
Let $\kappa_{1},\kappa_{2}>0$. Let further $r>2, r_{1}\geq2$ and $r_{2}\geq 0$ and assume that $r=r_{1}+r_{2}$. Then there exists a $c>0$ such that
\begin{enumerate}[(i)]
\item $\displaystyle \int_{0}^{\infty} M\big(\kappa_{1}(s^{2}\pm1)^{2}+\kappa_{2}\big)^{-\frac{1}{4}}
    L\big(\kappa_{1}(s^{2}\pm1)^{2}+\kappa_{2}\big)^{-r}\,ds
    \leq c\kappa_{1}^{-\frac{1}{4}}L(\kappa_{1})^{-r_{1}+1}L(\kappa_{2})^{-r_{2}}.$
\item $\displaystyle \int_{0}^{\infty} (s^{2}\pm1)^{-\frac{1}{2}}
    L\big(\kappa_{1}(s^{2}\pm1)^{2}\big)^{-r}\,ds
    \leq cL(\kappa_{1}^{-1})L(\kappa_{1})^{-r+1}.$
\item $\displaystyle \int_{0}^{\infty} M\big(\kappa_{1}s^{2}+\kappa_{2}\big)^{-\frac{1}{2}}
    L\big(\kappa_{1}s^{2}+\kappa_{2}\big)^{-r}\,ds
    \leq c\kappa_{1}^{-\frac{1}{2}}L(\kappa_{2})^{-r+2}.$
\item $\displaystyle \int_{0}^{\infty} \big(\kappa_{1}s^{2}+\kappa_{2}\big)^{-\frac{1}{2}}
    L\big(\kappa_{1}s^{2}+\kappa_{2}\big)^{-r}\,ds
    \leq c\kappa_{1}^{-\frac{1}{2}}L(\kappa_{2}^{-1})L(\kappa_{2})^{-r+2}.$
\end{enumerate}
\end{Lemma}

\begin{proof}
We start with (i) and first note that integral is smaller than or equal to $L(\kappa_{2})^{-r_{2}}I(\kappa_{1})$, where
\begin{align*}
I(\kappa_{1})
&:=\int_{0}^{\infty} M\big(\kappa_{1}(s^{2}-1)^{2}\big)^{-\frac{1}{4}}L\big(\kappa_{1}(s^{2}-1)^{2}\big)^{-r_{1}}\,ds\\
&=\frac{1}{2}\int_{-1}^{\infty}\big(t+1\big)^{-\frac{1}{2}}
        M\big(\kappa_{1}t^{2}\big)^{-\frac{1}{4}}L\big(\kappa_{1}t^{2}\big)^{-r_{1}}\,dt.
\end{align*}
Note that the integral in $I(\kappa_{1})$ is absolutely convergent for every $\kappa_{1}>0$ and the resulting function $I$ is continuous. We need to prove that
$$
I(\kappa_{1})\leq c \kappa_{1}^{-\frac{1}{4}}L(\kappa_{1})^{-r_{1}+1}
$$
for some $c>0$.
It is enough to consider small and large $\kappa_{1}$. First assume that $\kappa_{1}\leq 9$.  Then there exist $c, c', c''>0$ such that
\begin{align*}
I(\kappa_{1})
&\leq\frac{1}{2}\int_{-1}^{\infty}\big(t+1\big)^{-\frac{1}{2}}
        M\big(\kappa_{1}t^{2}\big)^{-\frac{1}{4}}L\big(\kappa_{1}t^{2}\big)^{-2}\,dt\\
&=c\int_{-1}^{\frac{3}{\sqrt{\kappa_{1}}}}\big(t+1\big)^{-\frac{1}{2}}
        \,dt
    +\kappa_{1}^{-\frac{1}{4}}\int_{\frac{3}{\sqrt{\kappa_{1}}}}^{\infty}\big(t+1\big)^{-\frac{1}{2}}
        t^{-\frac{1}{2}}\log\big(\kappa_{1}t^{2}\big)^{-2}\,dt\\
&\leq c'\kappa_{1}^{-\frac{1}{4}}
    +c'\kappa_{1}^{-\frac{1}{4}}\int_{\log(9)}^{\infty}u^{-2}\,du\\
&= c'' \kappa_{1}^{-\frac{1}{4}}.
\end{align*}
Now assume $\kappa_{1}\geq 9^{3}$. We define $\delta=\kappa^{-\frac{1}{3}}$. Then $\delta\leq9^{-1}$ and $\kappa_{1}t^{2}\leq \kappa_{1}^{\frac{1}{3}}$ if and only if $|t|\leq \delta$. Therefore, $|t|\geq \delta$ implies $\kappa_{1}t^{2}\geq \kappa_{1}^{\frac{1}{3}}\geq 9$ and we find that there exist $c,c'>0$ such that
\begin{align*}
I(\kappa_{1})
&\leq c\int_{0}^{\delta}|1-t|^{-\frac{1}{2}}\,dt\\
    &\qquad+2\kappa_{1}^{-\frac{1}{4}}\log(\kappa_{1}^{\frac{1}{3}})^{-r_{1}}\int_{\delta}^{1}|1-t|^{-\frac{1}{2}}t^{-\frac{1}{2}}\,dt\\
    &\qquad+\kappa_{1}^{-\frac{1}{4}}\int_{1}^{\infty}(1+t)^{-\frac{1}{2}}t^{-\frac{1}{2}}\log(\kappa_{1}t^{2})^{-r_{1}}\,dt\\
&\leq c'\delta+c'\kappa_{1}^{-\frac{1}{4}}\log(\kappa_{1})^{-r_{1}}\\
    &\qquad+\kappa_{1}^{-\frac{1}{4}}\int_{1}^{\infty}t^{-1}\log(\kappa_{1}t^{2})^{-r_{1}}\,dt.
\end{align*}
The latter is smaller than $c''\kappa_{1}^{-\frac{1}{4}}L(\kappa_{1})^{-r_{1}+1}$ for some $c''>0$ as
$$
\int_{1}^{\infty}t^{-1}\log(\kappa_{1}t^{2})^{-r_{1}}\,dt
=\frac{1}{2}\int_{\log(\kappa_{1})}^{\infty}s^{-r_{1}}\,ds
=\frac{1}{2(r_{1}-1)}\log(\kappa_{1})^{-r_{1}+1}.
$$
This proves (i).

In order to prove (ii), it suffices to consider the desired inequality only for the case with the minus signs in the integrand.
Let $\delta:=3\kappa_{1}^{-\frac{1}{2}}$.
Since $|s^{2}-1|\leq \delta$ if and only if $\kappa_{1}(s^{2}-1)^{2}\leq 9$ we have
\begin{align}\label{eq decomp of integral for (ii)}
\nonumber&\int_{0}^{\infty} (s^{2}-1)^{-\frac{1}{2}}L\big(\kappa_{1}(s^{2}-1)^{2}\big)^{-r}\,ds\\
\nonumber&\qquad=\kappa_{1}^{\frac{1}{4}}\int_{\{s\in[0,\infty):|s^{2}-1|\geq \delta\}}
            M\big(\kappa_{1}(s^{2}-1)^{2}\big)^{-\frac{1}{4}}L\big(\kappa_{1}(s^{2}-1)^{2}\big)^{-r}\,ds\\
&\qquad\qquad\qquad +c\int_{\{s\in[0,\infty):|s^{2}-1|\leq \delta\}}(s^{2}-1)^{-\frac{1}{2}}\,ds
\end{align}
with $c=\log(9)^{-r}$.
In view of (i) the first term on the right hand side of (\ref{eq decomp of integral for (ii)}) is smaller than or equal to $c'L(\kappa_{1})^{-r+1}$ for some $c'>0$.

Now we turn our attention to the integral in the second term on the right-hand side of (\ref{eq decomp of integral for (ii)}). Up to a constant it is equal to
$$
J(\kappa_{1})
:=\int_{-\min(1,\delta)}^{\delta}(t+1)^{-\frac{1}{2}}|t|^{-\frac{1}{2}}\,dt.
$$
Note that the integral is absolutely convergent and that the function $J$ is continuous. It suffices to prove that $J(\kappa_{1})\leq c L(\kappa_{1}^{-1})L(\kappa_{1})^{-r+1}$ for some $c>0$. It is enough to consider small and large $\kappa_{1}$.
First let $\kappa_{1}\leq 9$. Then there exists a $c>0$ such that
$$
J(\kappa_{1})
=\int_{-1}^{\delta}(t+1)^{-\frac{1}{2}}|t|^{-\frac{1}{2}}\,dt
\leq c L(\kappa_{1}^{-1}).
$$
Next, let $\kappa_{1}\gg 9$. Then there exists a $c>0$ such that
$$
J(\kappa_{1})
=\int_{-\delta}^{\delta}(t+1)^{-\frac{1}{2}}|t|^{-\frac{1}{2}}\,dt
\leq\frac{1}{2}\int_{-\delta}^{\delta}|t|^{-\frac{1}{2}}\,dt
\leq c\kappa_{1}^{-\frac{1}{4}}.
$$
This proves (ii).

By performing a substitution of variables $s'=\sqrt{\kappa_{1}}s$ we may reduce the proof of (iii) and (iv) to the case that $\kappa_{1}=1$. Since $L$ is an increasing function, $L\big(\kappa_{1}s^{2}+\kappa_{2}\big)\geq L(\kappa_{2})$, hence
$$
L\big(s^{2}+\kappa_{2}\big)^{-r}\leq
L\big(\kappa_{2}\big)^{-r+2}L\big(s^{2}+\kappa_{2}\big)^{-2}.
$$
Using that the integrand decreases as a function of $\kappa_{2}$ we find
$$
\int_{0}^{\infty} M\big(s^{2}+\kappa_{2}\big)^{-\frac{1}{2}}
    L\big(s^{2}+\kappa_{2}\big)^{-2}\,ds
\leq\int_{0}^{\infty} M(s^{2})^{-\frac{1}{2}}
    L(s^{2})^{-2}\,ds
<\infty.
$$
This proves (iii).
To prove (iv) it suffices to show that
$$
\int_{0}^{\infty} \big(s^{2}+\kappa_{2}\big)^{-\frac{1}{2}}
    L\big(s^{2}+\kappa_{2}\big)^{-2}\,ds
    \leq c L(\kappa_{2}^{-1})
$$
for some $c>0$.
We may assume that $\kappa_{2}<8$. Now the integral on the left-hand side is smaller than or equal to
$$
\int_{0}^{\sqrt{9-\kappa_{2}}} \big(s^{2}+\kappa_{2}\big)^{-\frac{1}{2}}L\big(s^{2}+\kappa_{2}\big)^{-2}\,ds
    +\int_{\sqrt{9-\kappa_{2}}}^{\infty} \big(s^{2}+\kappa_{2}\big)^{-\frac{1}{2}}L\big(s^{2}+\kappa_{2}\big)^{-2}\,ds.
$$
The second term is bounded by
$$
\int_{1}^{\infty} \big(s^{2}\big)^{-\frac{1}{2}}L\big(s^{2}\big)^{-2}\,ds
<\infty;
$$
the first term is equal to
\begin{align*}
c\int_{0}^{\sqrt{9-\kappa_{2}}} \big(s^{2}+\kappa_{2}\big)^{-\frac{1}{2}}\,ds
&=c\int_{0}^{\sqrt{\frac{9}{\kappa_{2}}-1}} \big(s^{2}+1\big)^{-\frac{1}{2}}\,ds\\
&=c\,\mathrm{arsinh}(\sqrt{\frac{9}{\kappa_{2}}-1})
\end{align*}
with $c=\log(9)^{-2}$.
This proves (iv) as $\mathrm{arsinh}(x)\sim\log(x)$ for $x\to\infty$.
\end{proof}

\begin{Prop}\label{Prop convergence and behavior of H_Q}
Let $r>4$.
\begin{enumerate}[(i)]
\item Assume that $\F=\R$. There exists a $c>0$ such that for every $t\in V$
$$
a_{t}^{\rho_{Q_{1}}}\int_{N_{Q_{1}}}\Phi(a_{t}n)^{-\frac{1}{4}}\big(\log\circ\Phi(a_{t}n)\big)^{-r}\,dn
\leq c \cosh(t_{1}-t_{2})^{-1}(1+\|t\|)^{-r+4}
$$
and
$$
a_{t}^{\rho_{Q_{3}}}\int_{N_{Q_{3}}}\Phi(a_{t}n)^{-\frac{1}{4}}\big(\log\circ\Phi(a_{t}n)\big)^{-r}\,dn
\leq c \cosh(t_{2}-t_{3})^{-1}(1+\|t\|)^{-r+4}.
$$
\item
Assume that $\F=\C$. There exist $c,C>0$ such that for every $t\in V$
\begin{align}
\label{eq estimate of a^rho int_N1 Phi^(-1/4)(1+log Phi)^(-r) complex item 1}
&\nonumber a_{t}^{\rho_{Q_{1}}}
    \int_{N_{Q_{1}}}\Phi(a_{t}n)^{-\frac{1}{2}}\big(\log\circ\Phi(a_{t}n)\big)^{-r}\,dn\\
&\qquad\qquad\leq c\cosh(t_{1}-t_{2})^{-1}L(e^{-t_{3}})L\big(e^{3t_{3}}\cosh(t_{1}-t_{2})\big)^{-r+4}\\
\label{eq estimate of a^rho int_N1 Phi^(-1/4)(1+log Phi)^(-r) complex item 2}
&\qquad\qquad\leq C L(e^{-t_{3}})^{r-3}\cosh(t_{1}-t_{2})^{-1}\big(1+|t_{1}-t_{2}|\big)^{-r+4}
\end{align}
and
\begin{align*}
&a_{t}^{\rho_{Q_{3}}}
    \int_{N_{Q_{3}}}\Phi(a_{t}n)^{-\frac{1}{2}}\big(\log\circ\Phi(a_{t}n)\big)^{-r}\,dn\\
&\qquad\qquad\leq c\cosh(t_{2}-t_{3})^{-1}L(e^{t_{1}})L\big(e^{-3t_{1}}\cosh(t_{2}-t_{3})\big)^{-r+4}\\
&\qquad\qquad\leq C L(e^{t_{1}})^{r-3}\cosh(t_{2}-t_{3})^{-1}\big(1+|t_{2}-t_{3}|\big)^{-r+4}.
\end{align*}
\end{enumerate}
\end{Prop}

\begin{proof}
We will prove the estimates for the parabolic subgroup $Q_{1}$; the proof for the estimates for $Q_{3}$ is similar.

Let $k=\dim_{\R}(\F)$.
Note that  $N_{Q_{1}}=\{n_{y,z}:y,z\in\F\}$, where
\begin{equation}\label{eq def n_(y,z)}
n_{y,z}
:=\left(
    \begin{array}{ccc}
      1 & 0 & z \\
      0 & 1 & y \\
      0 & 0 & 1 \\
    \end{array}
  \right)
  \qquad(y,z\in\F).
\end{equation}
Let $t\in V$. For $g=a_{t}n_{y,z}$ the right-hand side of (\ref{eq def Phi}) is equal to
$$
I_{t}
:=e^{k\frac{t_{1}+t_{2}-2t_{3}}{2}}
    \int_{\F}\int_{\F}\Phi(a_{t}n_{y,z})^{-\frac{k}{4}}\big(\log\circ\Phi(a_{t}n_{y,z})\big)^{-r}\,dz\,dy.
$$
Since $\Phi(x)\geq 9$ for all $x\in G$, see Remark \ref{Rem Phi geq 9}, we have
$$
\log\circ\Phi
=L\circ\Phi.
$$
A straightforward computation of the right-hand side of (\ref{eq def Phi}) shows that
\begin{align}\label{eq Phi expression}
\nonumber
&\Phi(a_{t}n_{y,z})\\
&\nonumber=\Big(
        e^{4t_{1}}(1-|z|^{2})^{2}
        +e^{4t_{2}}(1+|y|^{2})^{2}
        +e^{4t_{3}}
        +2e^{-2t_{1}}|y|^{2}
        +2e^{-2t_{2}}|z|^{2}
        +2e^{-2t_{3}}|y|^{2}|z|^{2}
\Big)\\
&\quad\times\Big(
        e^{-4t_{1}}
        +e^{-4t_{2}}
        +e^{-4t_{3}}(1+|y|^{2}-|z|^{2})^{2}
        +2e^{2t_{1}}|y|^{2}
        +2e^{2t_{2}}|z|^{2}
\Big).
\end{align}
Define
\begin{align*}
\Phi_{1}(t,y,z)
&=\big(
        e^{4t_{2}}(1+|y|^{2})^{2}
        +e^{4t_{3}}
        \big)
\big(
        e^{-4t_{1}}
        +e^{-4t_{2}}
        +e^{-4t_{3}}(1+|y|^{2}-|z|^{2})^{2}
\big)
\\
\Phi_{2}(t,y,z)
&=\big(
        e^{4t_{1}}(1-|z|^{2})^{2}
        +e^{4t_{3}}
        \big)
\big(
        e^{-4t_{1}}
        +e^{-4t_{2}}
        +e^{-4t_{3}}(1+|y|^{2}-|z|^{2})^{2}
\big).
\end{align*}
We can estimate $\Phi(a_{t}n_{y,z})$ from below by both $\Phi_{1}(t,y,z)$ and $\Phi_{2}(t,y,z)$.

Now assume that $\F=\R$.
We first use the estimate $\Phi(a_{t}n_{y,z})\geq \Phi_{1}(t,y,z)$. We perform the substitution of variables $z=\sqrt{y^{2}+1}v$ and thus obtain that there exists a constant $c_{1}>0$ such that
\begin{align*}
&I_{t}
\leq e^{\frac{t_{1}+t_{2}-2t_{3}}{2}}
    \int_{\R}\int_{\R}M(\Phi_{1}(y,z))^{-\frac{1}{4}}L\big(\Phi_{1}(t,y,z)\big)^{-r}\,dz\,dy\\
&=c_{1}e^{\frac{t_{1}+t_{2}-2t_{3}}{2}}
    \int_{0}^{\infty}\int_{0}^{\infty}
        (y^{2}+1)^{\frac{1}{2}}\\
  &\qquad\times
        M\Big[\big(
            e^{4t_{2}}(y^{2}+1)^{2}
            +e^{4t_{3}}
            \big)
            \big(
            e^{-4t_{1}}
            +e^{-4t_{2}}
            +e^{-4t_{3}}(y^{2}+1)^{2}(v^{2}-1)^{2}
            \big)\Big]^{-\frac{1}{4}}\\
 &\qquad\times
        L\Big[\big(
            e^{4t_{2}}(y^{2}+1)^{2}
            +e^{4t_{3}}
            \big)
            \big(
            e^{-4t_{1}}
            +e^{-4t_{2}}
            +e^{-4t_{3}}(y^{2}+1)^{2}(v^{2}-1)^{2}
            \big)\Big]^{-r}
    \,dv\,dy.
\end{align*}
We apply Lemma \ref{Lemma integral inequalities}(i) to the inner integral with
$$
\kappa_{1}=\big(e^{4(t_{2}-t_{3})}(y^{2}+1)^{2} +1\big)(y^{2}+1)^{2},\quad
\kappa_{2}=\big(e^{4t_{2}}(y^{2}+1)^{2}+e^{4t_{3}}\big)\big(e^{-4t_{1}}+e^{-4t_{2}}\big),
$$
and thus we see that for every $2\leq r_{1}\leq r$ and $r_{2}=r-r_{1}$ there exists a constant $c_{2}>0$ such that $I_{t}$ is smaller than or equal to
\begin{align*}
&c_{2}e^{\frac{t_{1}+t_{2}-2t_{3}}{2}}
    \int_{\R}\big(e^{4(t_{2}-t_{3})}(y^{2}+1)^{2} +1\big)^{-\frac{1}{4}}
        L\Big(\big(e^{4(t_{2}-t_{3})}(y^{2}+1)^{2} +1\big)(y^{2}+1)^{2}\Big)^{-r_{1}+1}\\
    &\qquad\qquad\qquad\qquad\qquad\qquad\times
        L\Big(\big(e^{4t_{2}}(y^{2}+1)^{2}+e^{4t_{3}}\big)\big(e^{-4t_{1}}+e^{-4t_{2}}\big)\Big)^{-r_{2}}\,dy.
\end{align*}
(Here we have neglected a factor of $(y^{2}+1)^{-\frac{1}{2}}$ in the integrand.)
We now first apply Lemma \ref{Lemma integral inequalities}(ii) with $r_{1}=r$ and $r_{2}=0$. Using that $L(x^{2})\sim L(x)$ for $x\to\infty$, we thus obtain that there exists $c_{3}, c_{4}>0$ such that
\begin{align*}
I_{t}
&\leq c_{3}e^{\frac{t_{1}+t_{2}-2t_{3}}{2}}
    \int_{\R}\big(e^{4(t_{2}-t_{3})}(y^{2}+1)^{2}\big)^{-\frac{1}{4}}
        L\Big(e^{2(t_{2}-t_{3})}(y^{2}+1)^{2}\Big)^{-r+1}\,dy\\
&\leq c_{4} e^{\frac{t_{1}-t_{2}}{2}}
    L\big(e^{2(t_{3}-t_{2})}\big)L\big(e^{2(t_{2}-t_{3})}\big)^{-r+2}.
\end{align*}
Secondly we take $r_{1}=2$ and $r_{2}=r-2$ and use that
$$
\big(e^{4t_{2}}(y^{2}+1)^{2}+e^{4t_{3}}\big)\big(e^{-4t_{1}}+e^{-4t_{2}}\big)
\geq\max\big(e^{4(t_{3}-t_{1})}+e^{4(t_{3}-t_{2})},(y^{2}+1)^{2}\big).
$$
This yields the existence of constants $c_{5},c_{6}>0$ such that $I_{t}$ is smaller than or equal to
\begin{align*}
&c_{5}e^{\frac{t_{1}+t_{2}-2t_{3}}{2}}L\Big(e^{4(t_{3}-t_{1})}+e^{4(t_{3}-t_{2})}\Big)^{-r+4}
    \int_{\R}\big(e^{4(t_{2}-t_{3})}(y^{2}+1)^{2} +1\big)^{-\frac{1}{4}} L\Big((y^{2}+1)^{2}\Big)^{-2}\,dy\\
&\qquad\qquad\leq c_{6} e^{\frac{t_{1}-t_{2}}{2}}L\Big(e^{4(t_{3}-t_{1})}+e^{4(t_{3}-t_{2})}\Big)^{-r+4}.
\end{align*}
(Here we have neglected a factor of $L\Big(e^{4(t_{2}-t_{3})}\big((y^{2}+1)^{2}+1\big)(y^{2}+1)^{2}\Big)^{-1}$ in the integrand.)
These two inequalities for $I_{t}$ imply the existence of a constant $c>0$ such that for every $t\in V$ with $t_{2}\geq t_{1}$
$$
I_{t}\leq ce^{\frac{t_{1}-t_{2}}{2}}(1+\|t\|)^{-r+4}.
$$

We now use the estimate $\Phi(a_{t}n_{y,z})\geq \Phi_{2}(t,y,z)$. We perform the substitution of variables $y=\sqrt{|z^{2}-1|}v$ and thus we obtain that there exists a constant $c_{1}>0$ such that $I_{t}$ is smaller than or equal to
\begin{align*}
&c_{1}e^{\frac{t_{1}+t_{2}-2t_{3}}{2}}
    \int_{0}^{\infty}\int_{0}^{\infty}
        |z^{2}-1|^{\frac{1}{2}}\\
  &\qquad\times
        M\Big[\big(
        e^{4t_{1}}(1-|z|^{2})^{2}
        +e^{4t_{3}}
        \big)
        \big(
        e^{-4t_{1}}
        +e^{-4t_{2}}
        +e^{-4t_{3}}(z^{2}-1)^{2}(1-v^{2})^{2}
        \big)\Big]^{-\frac{1}{4}}\\
 &\qquad\times
        L\Big[\big(
        e^{4t_{1}}(1-|z|^{2})^{2}
        +e^{4t_{3}}
        \big)
        \big(
        e^{-4t_{1}}
        +e^{-4t_{2}}
        +e^{-4t_{3}}(z^{2}-1)^{2}(1-v^{2})^{2}
        \big)\Big]^{-r}
    \,dv\,dz.
\end{align*}
We apply Lemma \ref{Lemma integral inequalities}(i) to the inner integral with
$$
\kappa_{1}=\big(e^{4(t_{1}-t_{3})}(1-|z|^{2})^{2}+1\big)(z^{2}-1)^{2},\quad
\kappa_{2}=\big(e^{4t_{1}}(1-|z|^{2})^{2}+e^{4t_{3}}\big)\big(e^{-4t_{1}}+e^{-4t_{2}}\big),
$$
and thus we see that for every $2\leq r_{1}\leq r$ and $r_{2}=r-r_{1}$ there exists a constant $c_{2}>0$ such that $I_{t}$ is smaller than or equal to
\begin{align*}
&c_{2}e^{\frac{t_{1}+t_{2}-2t_{3}}{2}}
    \int_{0}^{\infty}\big(e^{4(t_{1}-t_{3})}(1-|z|^{2})^{2}+1\big)^{-\frac{1}{4}}
        L\Big(\big(e^{4(t_{1}-t_{3})}(1-|z|^{2})^{2}+1\big)(z^{2}-1)^{2}\Big)^{-r_{1}+1}\\
    &\qquad\qquad\qquad\qquad\qquad\qquad\times
        L\Big(\big(e^{4t_{1}}(1-|z|^{2})^{2}+e^{4t_{3}}\big)\big(e^{-4t_{1}}+e^{-4t_{2}}\big)\Big)^{-r_{2}}\,dz.
\end{align*}
Applying Lemma \ref{Lemma integral inequalities}(ii) to the remaining integral as above, we obtain a constant $c_{3}>0$ such that $e^{\frac{t_{1}-t_{2}}{2}}I_{t}$ is smaller than or equal to
$$
c_{4}\min\Big(
        L(e^{2(t_{3}-t_{1})})L(e^{2(t_{1}-t_{3})})^{-r+2},
        L\big(e^{4(t_{3}-t_{1})}+e^{4(t_{3}-t_{2})}\big)^{-r+4}
    \Big).
$$
It follows that there exists a $c>0$ such that for every $t\in V$ with $t_{1}\geq t_{2}$
$$
I_{t}\leq ce^{\frac{t_{2}-t_{1}}{2}}(1+\|t\|)^{-r+4}.
$$
This proves (i).

Next, assume that $\F=\C$.
We first use the estimate $\Phi(a_{t}n_{y,z})\geq \Phi_{1}(t,y,z)$. After introducing polar coordinates and subsequently performing the substitution of variables $v=|z|^{2}-|y|^{2}-1$, $w=|y|^{2}+1$, we obtain that there exists a constant $c_{1}>0$ such that
\begin{align*}
&I_{t}
\leq e^{t_{1}+t_{2}-2t_{3}}
    \int_{\R}\int_{\R}M(\Phi_{1}(y,z))^{-\frac{1}{2}}L\big(\Phi_{1}(t,y,z)\big)^{-r}\,dz\,dy\\
&=c_{1}e^{t_{1}+t_{2}-2t_{3}}
    \int_{1}^{\infty}\int_{-w}^{\infty}
        M\Big[\big(
            e^{4t_{2}}w^{2}
            +e^{4t_{3}}
            \big)
            \big(
            e^{-4t_{1}}
            +e^{-4t_{2}}
            +e^{-4t_{3}}v^{2}
            \big)\Big]^{-\frac{1}{2}}\\
 &\qquad\qquad\qquad\qquad\times
        L\Big[\big(
            e^{4t_{2}}w^{2}
            +e^{4t_{3}}
            \big)
            \big(
            e^{-4t_{1}}
            +e^{-4t_{2}}
            +e^{-4t_{3}}v^{2}
            \big)\Big]^{-r}
    \,dv\,dw\\
&\leq 2c_{1}e^{t_{1}+t_{2}-2t_{3}}
    \int_{0}^{\infty}\int_{0}^{\infty}
        M\Big[\big(
            e^{4t_{2}}w^{2}
            +e^{4t_{3}}
            \big)
            \big(
            e^{-4t_{1}}
            +e^{-4t_{2}}
            +e^{-4t_{3}}v^{2}
            \big)\Big]^{-\frac{1}{2}}\\
 &\qquad\qquad\qquad\qquad\times
        L\Big[\big(
            e^{4t_{2}}w^{2}
            +e^{4t_{3}}
            \big)
            \big(
            e^{-4t_{1}}
            +e^{-4t_{2}}
            +e^{-4t_{3}}v^{2}
            \big)\Big]^{-r}
    \,dw\,dv.
\end{align*}
We apply Lemma \ref{Lemma integral inequalities}(iii) to the inner integral with $\kappa_{1}=e^{4(t_{2}-t_{1})}+1+e^{4(t_{2}-t_{3})}v^{2}$ and $\kappa_{2}=e^{4(t_{3}-t_{1})}+e^{4(t_{3}-t_{2})}+v^{2}$ and we thus find that there exists a constant $c_{2}>0$ such that
\begin{align*}
I_{t}
&\leq c_{2}e^{t_{1}+t_{2}-2t_{3}}
    \int_{0}^{\infty}
    \big(e^{4(t_{2}-t_{1})}+1+e^{4(t_{2}-t_{3})}v^{2}\big)^{-\frac{1}{2}}L\big(e^{4(t_{3}-t_{1})}+e^{4(t_{3}-t_{2})}+v^{2}\big)^{-r+2}\,dv\\
&=c_{2}e^{t_{1}-t_{2}}
    \int_{0}^{\infty}
    \big(e^{4(t_{3}-t_{1})}+e^{4(t_{3}-t_{2})}+v^{2}\big)^{-\frac{1}{2}}L\big(e^{4(t_{3}-t_{1})}+e^{4(t_{3}-t_{2})}+v^{2}\big)^{-r+2}\,dv.
\end{align*}
We may now apply Lemma \ref{Lemma integral inequalities}(iv) to the remaining integral with $\kappa_{1}=1$ and $\kappa_{2}=e^{4(t_{3}-t_{1})}+e^{4(t_{3}-t_{2})}$. It follows that there exists a $c>0$ such that
$$
I_{t}
\leq c e^{t_{1}-t_{2}} L\big((e^{4(t_{3}-t_{1})}+e^{4(t_{3}-t_{2})})^{-1}\big)L\big(e^{4(t_{3}-t_{1})}+e^{4(t_{3}-t_{2})}\big)^{-r+4}.
$$

Using the estimate $\Phi(a_{t}n_{y,z})\geq\Phi_{2}(t,y,z)$ and a similar computation we obtain that there exists a constant $c>0$ such that
$$
I_{t}
\leq ce^{t_{2}-t_{1}} L\big((e^{4(t_{3}-t_{1})}+e^{4(t_{3}-t_{2})})^{-1}\big)L\big(e^{4(t_{3}-t_{1})}+e^{4(t_{3}-t_{2})}\big)^{-r+4}.
$$
Now observe that
\begin{align*}
&L\big((e^{4(t_{3}-t_{1})}+e^{4(t_{3}-t_{2})})^{-1}\big)L\big(e^{4(t_{3}-t_{1})}+e^{4(t_{3}-t_{2})}\big)^{-r+4}\\
&\qquad= L\big(\big(2e^{6t_{3}}\cosh(2t_{1}-2t_{2})\big)^{-1}\big)L\big(2e^{6t_{3}}\cosh(2t_{1}-2t_{2})\big)^{-r+4 }\\
&\qquad\leq L\big(e^{-6t_{3}}\big)L\big(2e^{6t_{3}}\cosh(2t_{1}-2t_{2})\big)^{-r+4 }\\
&\qquad\leq cL\big(e^{-t_{3}}\big)L\big(e^{3t_{3}}\cosh(t_{1}-t_{2})\big)^{-r+4 }
\end{align*}
for some $c>0$. This proves the estimate (\ref{eq estimate of a^rho int_N1 Phi^(-1/4)(1+log Phi)^(-r) complex item 1}).
For the second inequality (\ref{eq estimate of a^rho int_N1 Phi^(-1/4)(1+log Phi)^(-r) complex item 2}) we note that there exists a $c'>0$ such that
$$
L\big(e^{3t_{3}}\cosh(t_{1}-t_{2})\big)
\geq c' L(e^{-t_{3}})^{-1} \big(1+|t_{1}-t_{2}|\big)
\qquad(t\in V).
$$
\end{proof}

%%% ----------------------------------------------------------------------
\subsection{Proof of Theorem \ref{Thm convergence and behavior for sigma parabolics} for maximal $\sigma$-parabolic subgroups}
\label{Section Proof for maximal parabolics}
%%% ----------------------------------------------------------------------
In this section we give the proof for Theorem \ref{Thm convergence and behavior for sigma parabolics} for a maximal $\sigma$-parabolic subgroup, which we here denote $Q$.
We write $L_{Q}$ for the $\sigma$-stable Levi subgroup of $Q$, i.e., $L_{Q}=\sigma(Q)\cap Q$. Let $P$ be a minimal $\sigma$-parabolic subgroup contained in $Q$ and let $Q=M_{Q}A_{Q}N_{Q}$ and $P=M_{P}A_{P}N_{P}$ be Langlands decompositions of $Q$ and $P$ respectively, such that $A_{P}$ and $A_{Q}$ are $\sigma$-stable  and $A_{Q}\subseteq A_{P}$. Let $R:=M_{Q}\cap P$. Then $R$ is a minimal $\sigma$-parabolic subgroup of $M_{Q}$.

We first list a number of conclusions that can be drawn from the calculations in Section \ref{Subsection estimates}.

\medbreak
\medbreak

\begin{Lemma}\label{Lemma enum convergence and behavior of H_Q}
\quad
\begin{enumerate}
\item[1a.]
    If $\F=\R$, then for every $r>4$ there exists a constant $c>0$ such that for every $a\in A_{P}$
    $$
    a^{\rho_{Q}}\int_{N_{Q}}\Phi(an)^{-\frac{k}{4}}\big(\log\Phi(an)\big)^{-r}\,dn
    \leq  c \Big(\sum_{w\in W_{M_{Q}}}a^{w\cdot\rho_{R}}\Big)^{-1}\big(1+\|\log(a)\|\big)^{-r+4}.
    $$
\item[1b.]
    If $\F=\C$, then for every $r>4$ there exists a constant $c>0$ such that for every $a\in A_{Q}$ and $b\in A_{P}\cap M_{Q}$
    \begin{align*}
    &a^{\rho_{Q}}\int_{N_{Q}}\Phi(abn)^{-\frac{k}{4}}\big(\log\Phi(abn)\big)^{-r}\,dn\\
    &\qquad\leq  c \Big(\sum_{w\in W_{M_{Q}}}b^{w\cdot\rho_{R}}\Big)^{-1}\big(1+\|\log(b)\|\big)^{-r+4}\|\log a\|^{r-3}.
    \end{align*}
\item[2.]
    For every $r>4$ there exists a constant $C>0$ such that for every $\phi\in\cC(G/H)$
    $$
    \sup_{k\in K}\sup_{a\in A}\Big|(1+\|\log(a)\|)^{-1}a^{\rho_{Q}}\Rt_{Q}\phi(ka)\Big|\leq C\mu_{1,r}(\phi),
    $$
\item[3a.]
    If $\F=\R$, then for every $r>0$ there exists a constant $C>0$ such that for every $\phi\in\cC(G/H)$
    $$
    \mu_{1,r}^{L_{Q}}\big(\Ht_{Q}\phi\big)\leq C\mu_{1,r+4}(\phi),
    $$
\item[3b.]
    If $\F=\C$, then for every $r>0$ there exists a constant $C>0$ such that for every $\phi\in\cC(G/H)$
    $$
    \sup_{a\in A_{Q}}\Big((1+\|\log(a)\|)^{-r-1}\mu_{1,r}^{M_{Q}}\big(\Ht_{Q}\phi(a\dotvar)\big)\Big)\leq C\mu_{1,r+4}(\phi),
    $$
\end{enumerate}
\end{Lemma}

\begin{proof}
It suffices to prove the lemma for only one parabolic subgroup in each $H$-conjugacy class of maximal $\sigma$-parabolic subgroups, hence by Proposition \ref{Prop H-conj classes of max sigma-psgs} we may assume that $Q=Q_{i}$ for some $i\in\{1,2,3,4\}$. Since $\Phi$ is invariant under composition with $\sigma$ and conjugation with $w_{0}$, it is in view of Proposition \ref{Prop relations between Q_i's} enough to prove the theorem for $Q=Q_{1}$ and $Q=Q_{3}$. The assertions now follow from Proposition \ref{Prop convergence and behavior of H_Q}.
\end{proof}

We now give the proof of Theorem \ref{Thm convergence and behavior for sigma parabolics} for maximal parabolic subgroups $Q$.

Since $N_{Q}H/H$ is closed in $G/H$, the Radon transform $\Rt_{Q}$ defines a continuous linear map
$$
C_{c}^{\infty}(G/H)\to C(G/(L_{Q}\cap H)N_{Q}).
$$
It follows from 1a and 1b in Lemma \ref{Lemma enum convergence and behavior of H_Q} that this map extends to a continuous linear map
$$
\Rt_{Q}:\cC(G/H)\to C(G/(L_{Q}\cap H)N_{Q})
$$
which for all $\phi\in\cC(G/H)$ and $g\in G$ is given by (\ref{eq def Rt_P}) with absolutely convergent integrals. Since $\Rt_{Q}$ is $G$-equivariant and the left regular representation of $G$ on $\cC(G/H)$ is smooth, $\Rt_{Q}$ in fact defines a continuous linear map $\cC(G/H)\to C^{\infty}(G/(L_{Q}\cap H)N_{Q})$. Moreover, for every $\phi\in\cC(G/H)$ and $u\in\cU(\fg)$
$$
u(\Rt_{Q}\phi)
=\Rt_{Q}(u\phi).
$$

Since $\delta_{Q}$ is a smooth and $L_{Q}\cap H$ invariant character on $L_{Q}$, $\Ht_{Q}$ is a continuous linear map $\cC(G/H)\to C^{\infty}(L_{Q}/L_{Q}\cap H)$. It follows from 2 in Lemma \ref{Lemma enum convergence and behavior of H_Q} that $\Ht_{Q}$ also defines a continuous linear map $\cC(G/H)\to \cC(L_{Q}/L_{Q}\cap H)'$ and thus $\Ht_{Q}$ is in fact a continuous map to $C^{\infty}_{\temp}(L_{Q}/L_{Q}\cap H)$.

Let $\phi\in\cC(G/H)$ and let $u\in \cU(\fl_{Q})$.
It follows from the Leibniz rule and the fact that $\delta_{Q}$ is a character on $L_{Q}$, that there exists a $v\in\cU(\fl_{Q})$ such that
$$
u(\Ht_{Q}\phi)
=u\big(\delta_{Q}(\Rt_{Q}\phi)\big|_{L_{Q}}\big)
=\delta_{Q}v(\Rt_{Q}\phi)\big|_{L_{Q}}
=\delta_{Q}\Rt_{Q}(v\phi)\big|_{L_{Q}}
=\Ht_{Q}(v\phi).
$$
If $\F=\R$, then 3a in Lemma \ref{Lemma enum convergence and behavior of H_Q} implies that
$$
\mu^{L_{Q}}_{u,r}\big(\Ht_{Q}\phi\big)
=\mu^{L_{Q}}_{1,r}\big(u\Ht_{Q}\phi\big)
=\mu^{L_{Q}}_{1,r}\big(\Ht_{Q}(v\phi)\big)
\leq C\mu_{1,4+r}(v\phi)
=C\mu_{v,4+r}(\phi).
$$
This proves that $\Ht_{Q}$ defines a continuous linear map $\cC(G/H)\to \cC\big(L_{Q}/(L_{Q}\cap H)\big)$.
If $\F=\C$, the same argument, with $L_{Q}$ replaced by $M_{Q}$ show that $\phi\mapsto \Ht_{Q}\phi\big|_{M_{Q}}$ defines a continuous linear map $\cC(G/H)\to \cC\big(M_{Q}/(M_{Q}\cap H)\big)$.

%%% ----------------------------------------------------------------------
\subsection{Proof of Theorem \ref{Thm convergence and behavior for sigma parabolics} for minimal $\sigma$-parabolic subgroups}
\label{Section Proof for minimal parabolics}
%%% ----------------------------------------------------------------------
In this section we give the proof for Theorem \ref{Thm convergence and behavior for sigma parabolics} in case $P$ is a minimal $\sigma$-parabolic subgroup.

Let $P$ be a minimal $\sigma$-parabolic subgroup and let $Q$ be a maximal $\sigma$-parabolic subgroup containing $P$. Let $P=M_{P}A_{P}N_{P}$ and $Q=M_{Q}A_{Q}N_{Q}$ be Langlands decompositions so that $A_{P}$ and $A_{Q}$ are $\sigma$-stable and $A_{Q}\subseteq A_{P}$.
Then $M_{Q}\cap H$ is a symmetric subgroup of $M_{Q}$.
The group $R:=P\cap M_{Q}$ is a minimal $\sigma$-parabolic subgroup of $M_{Q}$.

Note that $M_{Q}$ is isomorphic to the group $\{g\in \GL(2,\F):|\det g|=1\}$. We differentiate between four cases.
\begin{enumerate}[(a)]
\item $\F=\R$ and $M_{Q}\cap H$ is compact. Then $M_{Q}\cap H$ is a maximal compact subgroup of $M_{Q}$ and hence it is isomorphic to $\Orth(2)$.
\item $\F=\R$ and $M_{Q}\cap H$ is not compact. Then $M_{Q}\cap H$ is isomorphic to $\Orth(1,1)$.
\item $\F=\C$ and $M_{Q}\cap H$ is compact. Then $M_{Q}\cap H$ is a maximal compact subgroup of $M_{Q}$; it is isomorphic to $\U(2)$.
\item $\F=\C$ and $M_{Q}\cap H$ is not compact. Then $M_{Q}\cap H$ is isomorphic to $\U(1,1)$. In this case
    $$
    M_{Q}/M_{Q}\cap H\simeq \PSL(2,\C)/\PSU(1,1)\simeq \SO(3,1)_{e}/\SO(2,1)_{e}.
    $$
    as $\SO(3,1)_{e}$-homogeneous spaces. Note that the Schwartz space and the Radon transforms for $M_{Q}/M_{Q}\cap H$ considered as a homogeneous space for $\SO(3,1)_{e}$ coincide with the Schwartz space and the Radon transforms for $M_{Q}/M_{Q}\cap H$ considered as a homogeneous space for $M_{Q}$.
\end{enumerate}

For the cases (a) and (c) it is well known and for (b) and (d) it follows from
\cite[Theorem 5.1]{AndersenFlensted-Jensen_CuspidalDiscreteSeriesForProjectiveHyperbolicSpaces} that
the integral
$$
\int_{N_{R}}\phi(n)\,dn
$$
is absolutely convergent for every continuous function $\phi:M_{Q}/M_{Q}\cap H\to \C$ satisfying
$$
\sup_{k\in K\cap M_{Q}}\sup_{a\in A_{P}\cap M_{Q}}(a^{\rho_{R}}+a^{-\rho_{R}})(1+\|\log a\|)|\phi(ka)|<\infty.
$$
Moreover, if $r>1$ and $\phi$ satisfies
$$
\sup_{k\in K\cap M_{Q}}\sup_{a\in A_{P}\cap M_{Q}}(a^{\rho_{R}}+a^{-\rho_{R}})(1+\|\log a\|)^{r}|\phi(ka)|<\infty,
$$
then the function
$$
\Rt_{R}\phi(g):=\int_{N_{R}}\phi(gn)\,dn\qquad(g\in M_{Q})
$$
satisfies in the cases (a), (b) and (c)
$$
\sup_{k\in K\cap M_{Q}}\sup_{a\in A_{P}\cap M_{Q}}(1+\|\log a\|)^{r-1}a^{\rho_{R}}\Rt_{R}\phi(ka)<\infty.
$$
In case (d) such strong estimates do not hold, but we still have
$$
\sup_{k\in K\cap M_{Q}}\sup_{a\in A_{P}\cap M_{Q}}a^{\rho_{R}}\Rt_{R}\phi(ka)<\infty.
$$
These estimates are well known if $M_{Q}\cap H$ is a maximal compact subgroup of $M_{Q}$ and in the cases (b) and (d) they again follow from \cite[Theorem 5.1]{AndersenFlensted-Jensen_CuspidalDiscreteSeriesForProjectiveHyperbolicSpaces}.

In all cases the multiplication map
$$
N_{R}\times N_{Q}\to N_{P};\qquad(\nu,n)\to \nu n
$$
is a diffeomorphism with Jacobian equal to the constant function $1$.
Therefore
\begin{equation}\label{eq int_(N_P)=int_(N_R) int_(N_Q)}
\int_{N_{P}}\psi(gn)\,dn
=\int_{N_{R}}\int_{N_{Q}}\psi(\nu n)\,dn\,d\nu
\end{equation}
for every $\psi\in L^{1}(N_{P})$.

Since $\delta_{Q}\big|_{N_{R}}=1$, it follows from 1a and 1b in Lemma
\ref{Lemma enum convergence and behavior of H_Q} that for $r>5$ and all $g\in G$ the integral
$$
\int_{N_{P}}\Phi(gn)^{-\frac{k}{4}}\big(\log\Phi(gn)\big)^{-r}\,dn
$$
is absolutely convergent. Moreover, we have the following estimates.
\begin{enumerate}[(i)]
\item
    If $\F=\R$, then for every $r>5$ there exists a constant $c>0$ such that for every $a\in A_{P}$
    $$
    a^{\rho_{P}}\int_{N_{P}}\Phi(an)^{-\frac{k}{4}}\big(\log\Phi(an)\big)^{-r}\,dn
    \leq  c \big(1+\|\log(a)\|\big)^{-r+5}.
    $$
\item
    If $\F=\C$, then for every $r>5$ there exists a constant $c>0$ such that for every $a\in A_{Q}$ and $b\in A_{P}\cap M_{Q}$
    \begin{align*}
    a^{\rho_{P}}\int_{N_{P}}\Phi(abn)^{-\frac{k}{4}}\big(\log\Phi(abn)\big)^{-r}\,dn
    \leq  c \|\log a\|^{3}.
    \end{align*}
\end{enumerate}

It follows from these estimates that for every $\phi\in\cC(G/H)$ and $g\in G$ the integral
$$
\Rt_{P}\phi(g):=\int_{N_{P}}\phi(gn)\,dn
$$
is absolutely convergent and that the map $\Rt_{P}$ thus obtained is a continuous linear map from $\cC(G/H)$ to $C(G/(L_{P}\cap H)N_{P})$. Since $\Rt_{P}$ is equivariant and the left regular representation of $G$ on $\cC(G/H)$ is smooth, $\Rt_{P}$ in fact defines a continuous linear map $\cC(G/H)\to C^{\infty}(G/(L_{P}\cap H)N_{P})$.

By the estimates (i) and (ii) $\Ht_{Q}$ defines a continuous linear map $\cC(G/H)\to \cC(L_{P}/L_{P}\cap H)'$ and thus $\Ht_{P}$ is in fact a continuous map to $C^{\infty}_{\temp}(L_{P}/L_{P}\cap H)$.
Finally, if $\F=\R$, then it follows as in the proof for maximal $\sigma$-parabolic subgroups in Section \ref{Section Proof for maximal parabolics} that $\Ht_{P}$ defines a continuous linear map $\cC(G/H)\to \cC(L_{Q}/L_{Q}\cap H)$. The remaining assertion in case $\F=\C$ is trivial as $M_{P}$ is compact.
This ends the proof of Theorem \ref{Thm convergence and behavior for sigma parabolics}.

%%% ----------------------------------------------------------------------
\section{Kernels}
\label{Section rep theory}
%%% ----------------------------------------------------------------------
Throughout this section we assume that $\F=\R,\C$.

\subsection{The Plancherel decomposition and kernels of Radon transforms}
For a general reductive symmetric space a description
of the discrete series representations has been given by the first author \cite{Flensted_Jensen_DiscreteSeriesForSemisimpleSymmetricSpaces} and Matsuki and Oshima \cite{MatsukiOshima_DescriptionOfTheDiscreteSeriesForSemisimpleSymmetricSpaces}.
In our setting, it follows from the rank condition in \cite{MatsukiOshima_DescriptionOfTheDiscreteSeriesForSemisimpleSymmetricSpaces} that the space $G/H$ does not admit discrete series. For a maximal $\sigma$ parabolic subgroup $Q$, with Langlands decomposition $Q=M_{Q}A_{Q}N_{Q}$ such that $M_{Q}$ and $A_{Q}$ are $\sigma$-stable, the symmetric space $M_{Q}/(M_{Q}\cap H)$ is of rank $1$ and hence
admits discrete series representations if and only if $M_{Q}\cap H$ is not compact.

Delorme \cite{Delorme_FormuleDePlancherelPourLesEsPacesSymetriquesReductifs} and independently Van den Ban and Schlichtkrull \cite{vdBan&Schlichtkrull_ThePlancherelDecompositionForAReductiveSymmetricSpace_I}, \cite{vdBan&Schlichtkrull_ThePlancherelDecompositionForAReductiveSymmetricSpace_II} have given a precise description of the Plancherel decomposition of a general reductive symmetric space. It follows from these descriptions that in our setting the space $L^{2}(G/H)$ decomposes as
\begin{equation}\label{eq L^{2} decomp}
L^{2}(G/H)=L^{2}_{\cP_{\sigma}}(G/H)\oplus L^{2}_{\cQ_{\sigma}}(G/H),
\end{equation}
where $L^{2}_{\cS}(G/H)$ for $\cS=\cP_{\sigma},\cQ_{\sigma}$ is a $G$-invariant closed subspace that is unitarily equivalent to a direct integral of representations that are induced from a parabolic subgroup contained in $\cS$. To be more precise, if $S\in\cS$ and $=M_{S}A_{S}N_{S}$ is a Langlands decomposition of $S$ such that $A_{S}$ is $\sigma$-stable, then $L^{2}_{\cS}(G/H)$ is unitarily equivalent to a direct integral of representations $\Ind_{S}^{G}(\xi\otimes\lambda\otimes1)$, where $\lambda\in i\fa_{S}^{*}$ and $\xi\in\widehat{M_{S}}$ is so that for some $v\in N_{K}(\fa)$
$$
\Hom_{G}\Big(\xi, L^{2}\big(M_{S}/(M_{S}\cap vHv^{-1})\big)\Big)\neq\{0\},
$$
i.e., $\xi$ is equivalent to a discrete series representation for the space $M_{S}\cap vHv^{-1}$.

Intersecting both sides of (\ref{eq L^{2} decomp}) with $\cC(G/H)$ yields a decomposition
$$
\cC(G/H)
=\cC_{\cP_{\sigma}}(G/H)\oplus\cC_{\cQ_{\sigma}}(G/H),
$$
where $\cC_{\cP_{\sigma}}(G/H)=\cC(G/H)\cap L^{2}_{\cP_{\sigma}}(G/H)$ and $\cC_{\cQ_{\sigma}}(G/H)=\cC(G/H)\cap L^{2}_{\cQ_{\sigma}}(G/H)$.
For a $\sigma$-parabolic subgroup $P$, we denote the kernel of $\Rt_{P}$ in $\cC(G/H)$ by $\ker(\Rt_{P})$.

The aim of this section is to prove the following theorem.

\begin{Thm}\label{Thm kernels}
\quad\begin{enumerate}[(i)]
\item$\displaystyle \bigcap_{P\in\cP_{\sigma}}\ker(\Rt_{P})=\cC_{\cQ_{\sigma}}(G/H)$.
\item$\displaystyle \bigcap_{Q\in\cQ_{\sigma}}\ker(\Rt_{Q})=\{0\}$.
\end{enumerate}
\end{Thm}

\begin{Rem}
Let $P$ be a $\sigma$-parabolic subgroup. If $h\in H$ and $P'=hPh^{-1}$, then
$$
\ker(\Rt_{P})=\ker(\Rt_{P'}).
$$
Therefore,
$$
\bigcap_{P\in\cP_{\sigma}}\ker(\Rt_{P})=\bigcap_{i\in\{1,2,3\}}\ker(\Rt_{P_{i}}),\qquad
\bigcap_{Q\in\cQ_{\sigma}}\ker(\Rt_{Q})=\bigcap_{i\in\{1,2,3,4\}}\ker(\Rt_{Q_{i}}).
$$
\end{Rem}

We will prove the theorem by defining $\tau$-spherical Harish-Chandra transforms, relating them to $\tau$-spherical Fourier transforms as defined by Delorme in Section 3 of \cite{Delorme_FormuleDePlancherelPourLesEsPacesSymetriquesReductifs}, and then using the Plancherel theorem to conclude the assertions. This program is carried out in the remainder of section \ref{Section rep theory}.

%%% ----------------------------------------------------------------------
\subsection{The $\tau$-spherical Harish-Chandra transform}
%%% ----------------------------------------------------------------------
Let $(\tau,V_{\tau})$ be a finite dimensional representation of $K.$
We write $C^{\infty}(G/H\colon \tau)$ for the space of smooth $V_{\tau}$-valued functions $\phi$ on $G/H$ that satisfy the transformation property
$$
\phi(kx)
=\tau(k)\phi(x)\qquad(k\in K,x\in G/H).
$$
We further write $C_c^{\infty}(G/H\colon \tau)$ and $\cC(G/H\colon \tau)$ for the subspaces of $C^{\infty}(G/H\colon\tau)$ consisting of compactly supported functions and  Schwartz functions, respectively.

Let $W$ be the Weyl group of the root system of
$\fa$ in $\fg.$ Then
$$
W=N_{K}(\fa)/Z_{K}(\fa).
$$
For a subgroup $S$ of $G$, we define $W_{S}$ to be the subgroup of $W$ consisting
of elements that can be realized in $N_{K\cap S}(\fa)$.

Let $P$ be a $\sigma$-parabolic subgroup containing $A$ and let $P=M_{P}A_{P}N_{P}$ be a Langlands decomposition such that $A_{P}\subseteq A$.
We write $\cW_{P}$ for a choice of a set of representatives in $N_{K}(\fa)$ for the double quotient $W_{P}\bs W/ W_{H}$.

We denote by $\tau_{M_{P}}$ the restriction of $\tau$ to $M_{P}$ and define
$$
\cC_{M_{P}}(\tau)
:=\bigoplus_{v \in \cW_{P}}\;\cC\big(M_{P}/M_{P}\cap vHv^{-1},\tau_{M_{P}}\big),
$$
If $\psi\in\cC(M_{P},\tau)$, we write $\psi_{v}$ for the component of $\psi$ in the space  $\cC\big(M_{P}/M_{P}\cap vHv^{-1},\tau_{M_{P}}\big)$.
For $v\in\cW$ we define the parabolic subgroup $P^v$ by
$$
P^{v}: = v^{-1}Pv.
$$
It follows from Theorem \ref{Thm convergence and behavior for sigma parabolics} that for every $\phi\in\cC(G/H,\tau)$, $a\in A_{P}$ and $m\in M_{P}$ the integral
$$
\Ht_{P,\tau}^{v}\phi(a)(m)
:=a^{\rho_{P}}\int_{N_{P^{v}}}\phi(mavn)\,dn
$$
is absolutely convergent. Moreover, for every $a\in A_{P}$ the function $\Ht_{P,\tau}^{v}\phi(a)$ belongs to $\cC(M_{P}/M_{P}\cap vHv^{-1},\tau_{M_{P}})$, and from 3a and 3b in Lemma \ref{Lemma enum convergence and behavior of H_Q} it follows that the map $A_{P}\to \cC(M_{P}/M_{P}\cap vHv^{-1},\tau_{M_{P}})$ thus obtained is continuous and tempered in the sense that for every continuous seminorm $\mu$ on $\cC(M_{P}/M_{P}\cap vHv^{-1},\tau_{M_{P}})$ there exists an $r>0$ so that
$$
\sup_{a\in A_{P}}(1+\|a\|)^{-r}\mu\big(\Ht_{P,\tau}^{v}\phi(a)\big)<\infty.
$$

\begin{Defi}
\label{d: tau HCT}
For a function $\phi\in \cC(G/H,\tau)$ we define its $\tau$-spherical Harish-Chandra transform  $\Ht_{P,\tau}\phi$ to be the function $A_{P}\to \cC(M_{P},\tau)$ given by
$$
\Big(\Ht_{P,\tau}\phi(a)\Big)_{\!v}(m)
:=a^{\rho_{P}}\int_{N_{P^{v}}}\phi(mavn)\,dn
\qquad\big(v\in\cW_{P}, m\in M_{P},a\in A_{P}\big).
$$
\end{Defi}

%%% ----------------------------------------------------------------------
\subsection{The $\tau$-spherical Fourier transform}
%%% ----------------------------------------------------------------------
We continue with the assumptions and notation from the previous section. Let $\cA_{2}(M_{P},\tau)$ be the subspace of $\cC(M_{P},\tau)$ consisting of all elements $\psi\in\cC(M_{P},\tau)$ such that for every $v\in\cW_{P}$ the function $\psi_{v}$ is finite under the action of the algebra $\cD(M_{P}/M_{P}\cap vHv^{-1})$ of $M_{P}$-invariant differential operators on $M_{P}/M_{P}\cap vHv^{-1}$.
For each $v\in\cW_{P}$ the space
$$
\cA(M_{P}/M_{P}\cap vHv^{-1},\tau_{M_{P}})
:=\{\psi_{v}:\psi\in\cA_{2}(M_{P},\tau)\}
$$
is finite dimensional. See \cite[Theorem 12, p. 312]{Varadarajan_HarmonicAnalysisOnRealReductiveGroups}. Equipped with the
restriction of the inner product of $L^2(M_{P}/M_{P}\cap vHv^{-1}, V_{\tau})$ this space is therefore a Hilbert space. Since $\cA_{2}(M_{P},\tau)$ is the direct sum of the spaces $\cA(M_{P}/M_{P}\cap v^{-1}Hv,\tau_{M_{P}})$, it is the finite direct sum of finite dimensional Hilbert spaces, and thus it is itself a finite dimensional Hilbert space.

For $\psi\in\cA_{2}(M_{P},\tau)$ and $\lambda\in\fa_{P,\C}^{*}$ let $\psi_{\lambda}:G/H\to V_{\tau}$ be the function given by
$$
\psi_{\lambda}(x)
:=\left\{
         \begin{array}{ll}
              0, & \big(x\notin \bigcup_{v\in \cW_{P}}PvH\big) \\
              a^{-\lambda+\rho_{P}}\psi_{v}(m), & \big(x\in N_{P}amvH, a\in A_{P}, m\in M_{P}, v\in \cW_{P}\big).
              \end{array}
          \right.
$$
If $\Re(\lambda-\rho_{P})$ is strictly $P$-dominant we define the (unnormalized) Eisenstein integral $E(P,\psi,\lambda):G/H\to V_{\tau}$ by
$$
E(P,\psi,\lambda)(x)
=\int_{K}\tau(k^{-1})\psi_{\lambda}(kx)\,dk
\qquad(x\in G/H)
$$
and for other $\lambda\in\fa_{P,\C}^{*}$ we define $E(P,\psi,\lambda)$ by meromorphic continuation. (See \cite[Section 3]{CarmonaDelorme_TransformationsDeFourier}. The Eisenstein integral can be normalized by setting
$$
E^{0}(P,\psi,\lambda):=E(P,C_{P|P}(1,\lambda)^{-1}\psi,\lambda)
$$
as an identity of meromorphic functions in $\lambda$. Here $C_{P|P}(1,\lambda)\in\End\big(\cA_{2}(M_{P},\tau)\big)$ is the $c$-function determined by the asymptotic expansion in \cite[(3.3)]{Delorme_FormuleDePlancherelPourLesEsPacesSymetriquesReductifs} for $E(P,\psi,\lambda)$ and is invertible for generic $\lambda$. The function $\lambda\mapsto C_{P|P}(1,\lambda)$ is meromorphic. In fact there exist a finite product $b=\prod_{j=1}^{n}(\langle \alpha_{j},\dotvar\rangle-c_{j})$ of factors $\langle \alpha_{j},\dotvar\rangle-c_{j}$, where $\alpha_{j}\in\Sigma$ does not vanish on $\fa_{P}$ and $c_{j}\in\C$, with the property that $\lambda\mapsto b(\lambda)C_{P|P}(1,\lambda)$ is holomorphic on an open neighborhood of $i\fa_{P}^{*}$. See \cite[Th\'eor\`eme 1]{CarmonaDelorme_TransformationsDeFourier}.

For $\phi\in C_{c}^{\infty}(G/H,\tau)$ let $\Ft_{P,\tau}^{0}\phi(\lambda)$ be the element of $\cA_{2}(M_{P},\tau)$ determined by
$$
\langle\Ft_{P,\tau}^{0}\phi(\lambda),\psi\rangle
=\int_{G/H}\big\langle\phi(x),E^{0}(P,\psi,\lambda)\big\rangle_{\tau}\,dx.
$$
By \cite[Th\'eor\`eme 4]{CarmonaDelorme_TransformationsDeFourier} $\Ft_{P,\tau}^{0}$ defines a continuous map $\cC(G/H:\tau)\to \sS(i\fa_{P}^{*})\otimes\cA_{2}(M_{P},\tau)$. The map $\Ft_{P,\tau}^{0}$ thus obtained is called the (normalized) $\tau$-spherical Fourier transform.

Let $\Ft_{A_{P}}$ be the euclidean Fourier transform on $A_{P}$, i.e.,
the transform $\Ft_{A_{P}}:\sS(A_{P})\to\sS(i\fa_{P}^{*})$ given by
$$
\Ft_{A_{P}}f(\lambda)=\int_{A_{P}}f(a)a^{-\lambda}\,da
\qquad(f\in\sS(A_{P}), \lambda\in i\fa_{P}^{*}).
$$
The continuous extension of $\Ft_{A_{P}}$ to a map $\sS'(A_{P})\to\sS'(i\fa_{P}^{*})$ we also denote by $\Ft_{A_{P}}$.

\begin{Lemma}\label{Lemma relation Ft_P and Ht_P}
Let $\phi\in \cC(G/H,\tau)$. Then
$$
\langle\Ft_{P,\tau}^{0}\phi(\lambda),\psi\rangle
=\Ft_{A_{P}}\Big(\langle\Ht_{P,\tau}\phi(\dotvar),C_{P|P}(1,\lambda)^{-1}\psi\rangle\Big)(\lambda)
\qquad\big(\lambda\in i\fa_{P}^{*},\psi\in\cA_{2}(M_{P},\tau)\big).
$$
as an identity of tempered distributions on $i\fa_{P}^{*}$.
\end{Lemma}

\begin{proof}
Let
$$
\Gamma_{P}
:=\sum_{\substack{\alpha\in\Sigma(\fg,\fa)\\\fg_{\alpha}\subseteq\fn_{P}}}\R_{\geq0}H_{\alpha},
$$
where $H_{\alpha}\in\fa^{*}$ is the element so that
$$
\alpha(Y)=\langle Y,H_{\alpha}\rangle\qquad(Y\in\fa).
$$
Let $B\subseteq \fa$ be so that $\supp(\phi)\subseteq K\exp(B)H$. It follows from \cite[Corollary 4.2]{Kuit_SupportTheorem} that $\supp\big(\Ht_{P,\tau}\phi\big)\subseteq \exp\big((B+\Gamma_{P})\cap\fa_{P}\big)$. Since $\Ht_{P,\tau}\phi$ is a tempered $\cC(M_{P}/M_{P}\cap vHv^{-1},\tau_{M_{P}})$-valued function, it follows that for every $t\in (0,1]$ the function
$$
a\mapsto a^{-t\rho_{P}}\Ht_{P,\tau}\phi(a)
$$
belongs to $\sS(A_{P})$ and the map
$$
[0,1]\to\sS'(A_{P});\qquad t\mapsto (\dotvar)^{-t\rho_{P}}\Ht_{P}\phi
$$
is continuous. The proof is analogous to the proof of \cite[Lemma 5.7]{Kuit_SupportTheorem}. For every $\psi\in\cA_{2}(M_{P},\tau)$, $\lambda\in i\fa_{P}^{*}$ and $t>0$
\begin{align*}
&\Ft_{A_{P}}\Big(\langle\Ht_{P,\tau}\phi(\dotvar),\psi\rangle\Big)(t\rho_{P}+\lambda)\\
&\quad=\int_{A_{P}}a^{-\lambda+(1-t)\rho_{P}}
    \int_{N_{P}}
    \sum_{v\in\cW_{P}}\int_{M_{P}/M_{P}\cap vHv^{-1}}
    \langle\phi(manvH),\psi_{v}(m)\rangle_{\tau}
    \,dm\,dn\,da\\
&\quad=\sum_{v\in\cW_{P}}\int_{M_{P}/M_{P}\cap vHv^{-1}}
    \int_{A_{P}}
    \int_{N_{P}}
    \langle\phi(manvH),\psi_{t\rho_{P}+\lambda}(manvH)\rangle_{\tau}
    \,dm\,dn\,da\\
&\quad=\int_{G/H}\langle\phi(x),\psi_{t\rho_{P}+\lambda}(x)\rangle_{\tau}\,dx.
\end{align*}
Since the measure on $G/H$ is invariant, the last integral is equal to
\begin{align*}
&\int_{K}\int_{G/H}\langle\phi(kx),\psi_{t\rho_{P}+\lambda}(kx)\rangle_{\tau}\,dx\,dk\\
&\qquad=\int_{G/H}\int_{K}\langle\phi(x),\tau(k^{-1})\psi_{t\rho_{P}+\lambda}(kx)\rangle_{\tau}\,dk\,dx\\
&\qquad=\int_{G/H}\langle\phi(x),E(P,\psi,t\rho_{P}+\lambda)\rangle_{\tau}\,dx.
\end{align*}
After replacing $\psi$ by $C_{P|P}(1,t\rho_{P}+\lambda)^{-1}\psi$, we have thus obtained the identity
$$
\langle\Ft_{P,\tau}^{0}\phi(t\rho_{P}+\lambda),\psi\rangle
=\Ft_{A_{P}}\Big(\langle\Ht_{P,\tau}\phi(\dotvar),C_{P|P}(1,t\rho_{P}+\lambda)^{-1}\psi\rangle\Big)(t\rho_{P}+\lambda).
$$
The assertion in the lemma now follows by taking the limit for $t\downarrow0$ on both sides of the equation.
\end{proof}

%%% ----------------------------------------------------------------------
\subsection{Proof of Theorem \ref{Thm kernels}}
%%% ----------------------------------------------------------------------
For a $K$-invariant subspace $V$ of $L^{2}(G/H)$ and a finite subset $\vartheta$ of $\widehat{K}$, we denote by $V_\vartheta$  the subspace of $V$ of all $K$-finite vectors with isotypes contained in $\vartheta$.
By continuity and equivariance of the Radon transforms, it suffices to prove that for all finite subset $\vartheta$ of $\widehat{K}$
$$
\bigcap_{P\in\cP_{\sigma}}\ker(\Rt_{P})_{\vartheta}=\cC_{\cQ_{\sigma}}(G/H)_{\vartheta},
\qquad
\bigcap_{Q\in\cQ_{\sigma}}\ker(\Rt_{Q})_{\vartheta}=\{0\}.
$$
Let $C(K)_\vartheta$ be the space of $K$-finite continuous functions
on $K,$ whose right $K$-types belong to $\vartheta$
and let $\tau$ denote the right
regular representation of $K$ on $V_{\tau}:= C(K)_{\vartheta}$. Then the canonical map
$$
\varsigma:\cC(G/H)_{\vartheta}\to\cC(G/H:\tau)
$$
given by
$$
\varsigma\phi(x)(k)=\phi(kx)
\qquad\big(\phi\in\cC(G/H)_{\vartheta}, k\in K, x\in G/H\big)
$$
is a linear isomorphism.
Let $\phi\in\cC(G/H)_{\vartheta}$ and let $P$ be a $\sigma$-parabolic subgroup. Then $\Rt_{P^{v}}\phi=0$ for every $v\in\cW_{P}$ if and only if $\Ht_{P,\tau}(\varsigma\phi)=0$.
We note that $\varsigma\big(\cC_{P}(G/H)_{\vartheta}\big)$ equals the space $\cC(G/H,\tau)^{(P)}$ defined in Th\'eor\`eme 2 in \cite{Delorme_FormuleDePlancherelPourLesEsPacesSymetriquesReductifs}. It follows from Lemma \ref{Lemma relation Ft_P and Ht_P} and the fact that the c-function $C_{P|P}(1,\lambda)$ is invertible for generic $\lambda$ that
$$
\ker(\Ht_{P,\tau})\subseteq\ker(\Ft_{P,\tau}^{0})
$$
with equality if $P$ is a minimal $\sigma$-parabolic subgroup.

Note that $\Ft_{G,\tau}^{0}=\{0\}$ since $G/H$ does not admit discrete series representation. It now follows from \cite[Th\'eor\`eme 2]{Delorme_FormuleDePlancherelPourLesEsPacesSymetriquesReductifs} that if $P\in\cP_{\sigma}$ and $Q\in\cQ_{\sigma}$, then
$$
\ker(\Ht_{P,\tau})=\cC(G/H,\tau)^{(Q)},
\qquad
\ker(\Ht_{Q,\tau})\subseteq\cC(G/H,\tau)^{(P)}.
$$
Moreover, the identity (\ref{eq int_(N_P)=int_(N_R) int_(N_Q)}) implies that $\ker(\Ht_{Q,\tau})\subseteq\ker(\Ht_{P,\tau})$. The proof for Theorem \ref{Thm kernels} is now concluded by noting that
$$
\cC(G/H,\tau)^{(Q)}\cap\cC(G/H,\tau)^{(P)}
=\{0\},
$$
so that $\ker(\Ht_{Q,\tau})=\{0\}$.

%%% ----------------------------------------------------------------------
\section{Divergence of cuspidal integrals for $\SL(3,\H)/\Sp(1,2)$}
\label{Section Divergence}
%%% ----------------------------------------------------------------------
In this section, let $\F$ be equal to $\H$.
For $y,z\in\H$, let $n_{y,z}$ be given by (\ref{eq def n_(y,z)}) and let $\Phi$ be given by (\ref{eq def Phi}).

\begin{Lemma}\label{Lemma divergence of H-integral}
Let $\epsilon<\frac{1}{6}$. Then the integral
$$
\int_{\H}\int_{\H}\Phi(n_{y,z})^{-1-\epsilon}\,dy\,dz
$$
is divergent.
\end{Lemma}

\begin{proof}
It follows from (\ref{eq Phi expression}) that $\Phi(n_{y,z})$ is equal to
$$
\Big(
        (1-|z|^{2})^{2}
        +(1+|y|^{2})^{2}
        +1
        +2|y|^{2}
        +2|z|^{2}
        +2|y|^{2}|z|^{2}
\Big)
\Big(
        2
        +(1+|y|^{2}-|z|^{2})^{2}
        +2|y|^{2}
        +2|z|^{2}
\Big).
$$
Let $R>1$. There exists a constant $c>0$ such that for every $y,z\in\H$ satisfying $1<|y|<|z|<|y|+1$ we have
$$
        (1-|z|^{2})^{2}
        +(1+|y|^{2})^{2}
        +1
        +2|y|^{2}
        +2|z|^{2}
        +2|y|^{2}|z|^{2}
\leq
c|y|^{4}
$$
and
$$
        2
        +(1+|y|^{2}-|z|^{2})^{2}
        +2|y|^{2}
        +2|z|^{2}
\leq c|y|^{2}.
$$
Therefore there exists a constant $C>0$ such that
$$\int_{\H}\int_{\H}\Phi(n_{y,z})^{-1-\epsilon}\,dy\,dz
\geq\int_{\substack{y\in\H\\|y|>R}}\int_{\substack{z\in\H\\|y|<|z|<|y|+1}}\Phi(n_{y,z})^{-1-\epsilon}\,dy\,dz
\geq C\int_{R}^{\infty}r^{6}r^{-6-6\epsilon}\,dr.
$$
The latter integral is divergent for $\epsilon<\frac{1}{6}$.
\end{proof}

\begin{Prop}\label{Prop divergence for H}
There exists a positive function $\phi\in\cC(G/H)$ such that for every $Q\in\cQ_{\sigma}$ that is $H$-conjugate to $Q_{1}$ or $Q_{4}$ and for every $g\in G$ the integral
$$
\int_{N_{Q}}\phi(gn)\,dn
$$
is divergent.
\end{Prop}

\begin{proof}
Let $\Xi$ be Harish-Chandra's bi-$K$-invariant spherical function $\phi_{0}$ on $G$. Define $\Theta:G/H\to\R_{+}$ by
$$
\Theta(x)=\sqrt{\Xi(g\sigma(g)^{-1})}\qquad(x\in G/H).
$$
Let $0<\epsilon<\frac{1}{6}$ and let $\phi=\Theta^{-1-\epsilon}$.
It follows from Propositions 17.2 and 17.3 in \cite{vdBan_PrincipalSeriesII} that $\phi\in\cC(G/H)$ and that for every $g\in G$ there exists a $c>0$ such that
$$
\phi(gx)\geq c\Phi(x)^{-1-\epsilon}\qquad(x\in G/H).
$$
The proposition now follows from the Proposition \ref{Prop relations between Q_i's} and Lemma \ref{Lemma divergence of H-integral}.
\end{proof}

\begin{Cor}
Let $\phi\in\cC(G/H)$ be as in Proposition \ref{Prop divergence for H}. For every $P\in\cP_{\sigma}$ and every $g\in G$ the integral
$$
\int_{N_{P}}\phi(gn)\,dn
$$
is divergent.
\end{Cor}

\begin{proof}
Note that $\fn_{Q_{1}}\subset\fp_{1}$. Let $g\in G$. By Tonelli's theorem
$$
\int_{N_{P_{1}}}\phi(gn)\,dn
=\int_{N_{P_{1}}/N_{Q_{1}}}\int_{N_{Q_{1}}}\phi(gnn_{1})\,dn_{1}\,dn.
$$
It follows from Proposition \ref{Prop divergence for H} and Fubini's theorem that the right-hand side is divergent.
Similarly we have $\fn_{Q_{4}}\subset\fp_{2}\cap\fp_{3}$ and thus we see that the integrals for $P=P_{2}$ and $P=P_{3}$ are divergent as well. The assertion now follows from Proposition \ref{Prop H-conj classes of min sigma-psgs}.
\end{proof}

%%% ----------------------------------------------------------------------
 \newcommand{\SortNoop}[1]{}\def\dbar{\leavevmode\hbox to 0pt{\hskip.2ex
  \accent"16\hss}d}

%%%-----------------------------------------------------------------------

%%%-----------------------------------------------------------------------
\def\adritem#1{\hbox{\small #1}}
\def\distance{\hbox{\hspace{3.5cm}}}
\def\apetail{@}
\def\addFlenstedJensen{\vbox{
\adritem{M.~Flensted--Jensen}
\adritem{Department of Mathematical Sciences}
\adritem{University of Copenhagen}
\adritem{Universitetsparken 5}
\adritem{2100 Copenhagen \O}
\adritem{Denmark}
\adritem{E-mail: mfj{\apetail}math.ku.dk}
}
}
\def\addKuit{\vbox{
\adritem{J.~J.~Kuit}
\adritem{Institut f\"ur Mathematik}
\adritem{Universit\"at Paderborn}
\adritem{Warburger Stra{\ss}e 100}
\adritem{33089 Paderborn}
\adritem{Germany}
\adritem{E-mail: j.j.kuit{\apetail}gmail.com}
}
}
\mbox{}
\vfill
\hbox{\vbox{\addFlenstedJensen}\vbox{\distance}\vbox{\addKuit}}
%%% ----------------------------------------------------------------------
\end{document}